%% file: main.tex
\documentclass{article}

\usepackage{arxiv}

\usepackage[utf8]{inputenc} 
\usepackage[T1]{fontenc}    
\usepackage{url}            
\usepackage{booktabs}       
\usepackage{amsfonts}       
\usepackage{nicefrac}       
\usepackage{microtype}      
\usepackage{lipsum}
\usepackage{graphicx}
\graphicspath{ {./images/} }

\usepackage[utf8]{inputenc}
\usepackage[english]{babel}
\usepackage[ruled,lined]{algorithm2e}
\SetKwRepeat{Do}{do}{while}

\usepackage[numbers]{natbib}
\usepackage{graphicx}
\usepackage[font=footnotesize,labelfont=bf]{caption}
\usepackage{subcaption}
\usepackage{amsmath,amsthm}
\usepackage{amssymb}
\usepackage{mathtools}
\usepackage{appendix}
\usepackage{amssymb}
\usepackage{wrapfig}
\usepackage[colorlinks, linkcolor = blue, urlcolor=blue]{hyperref}

\DeclareMathOperator*{\argmin}{argmin}

\usepackage[T1]{fontenc}
\usepackage{makecell}

\usepackage{sidecap}
\usepackage{comment}
\usepackage{multirow}
\usepackage{booktabs}
\usepackage[table,xcdraw]{xcolor}
\usepackage{placeins}

\usepackage{xcolor}

\numberwithin{equation}{section}
\newtheorem{theorem}{Theorem}

\newtheorem{definition}{Definition}
\newtheorem{lemma}{Lemma}
\newtheorem{claim}{Claim}

\title{An Optimization Framework for Power Infrastructure Planning}

\author{
  Nina Wiedemann \\
  Institute for Operations Research\\
  ETH Zurich\\
  Switzerland \\
  \texttt{nina.wiedemann@alumni.ethz.ch} \\
  \And
 David Adjiashvili \\
  Institute for Operations Research \\
  ETH Zurich\\
  Switzerland \\
  \texttt{david.adjiashvili@ifor.math.ethz.ch} \\
}

\DontPrintSemicolon

\begin{document}
\maketitle
\begin{abstract}
    The ubiquitous expansion and transformation of the energy supply system involves large-scale power infrastructure construction projects. In the view of investments of more than a million dollars per kilometre, planning authorities aim to minimise the resistances posed by multiple stakeholders. Mathematical optimisation research offers efficient algorithms to compute globally optimal routes based on geographic input data. \\
    We propose a framework that utilizes a graph model where vertices represent possible locations of transmission towers, and edges are placed according to the feasible distance between neighbouring towers. 
    In order to cope with the specific challenges arising in linear infrastructure layout, we first introduce a variant of the Bellman-Ford algorithm that efficiently computes the \textit{minimal-angle} shortest path. Secondly, an iterative procedure is proposed that yields a locally optimal path at considerably lower memory requirements and runtime. Third, we discuss and analyse methods to output $k$ diverse path alternatives. \\
    Experiments on real data show that compared to previous work, our approach reduces the resistances by more than 10\% in feasible time, while at the same time offering much more flexibility and functionality.
    Our methods are demonstrated in a simple and intuitive graphical user interface, and an open-source package (LION) is available at  \url{https://pypi.org/project/lion-sp}.
\end{abstract}

\input{chapter/01introduction}
\input{chapter/02setup}
\input{chapter/03_angle}
\input{chapter/035_pipeline}
\input{chapter/04_ksp}

\input{chapter/05_experiments}
\input{chapter/06_discussion}

\section*{Acknowledgments}

We are very grateful for the fruitful collaboration with Gilytics AG, and want to deeply thank them for their support in collecting the relevant data and understanding the background of the problem, and for their extremely helpful feedback on our methods. Special thanks to  Stefano Grassi for his contribution.\\
Finally, we also want to thank Rico Zenklusen for his support, as well as Jordi Castells for his great help on improving the LION package.

\bibliography{references}
\vspace{42em}
\pagebreak
\section*{\huge Appendices}
\begin{appendices}
\input{chapter/appendix}
\end{appendices}

\end{document}

%% file: chapter/01introduction.tex
\section{Introduction}

Power infrastructure layout is a challenge of significant impact on society. Besides ensuring consistent and robust energy supply, it also plays a major role for the successful transition to renewable energies in view of climate change.
The European Union, for example, identified necessary investments in the EU-transmission networks 
of around 50,000 km overall length just from 2014 to 2024, of which 80\% directly or indirectly target the integration of renewable energies~\cite{feix2014netzentwicklungsplan}. With an estimated cost of 1.5 million Euros per km~\cite{tsos2015netzentwicklungsplan}, this leads to investments of approximately 75 billion Euros over ten years. In light of this huge scale, it is more important than ever to optimize the planning process. But while planning authorities aim at minimal costs, environmentalists demand to consider the impact on conservation areas and bird mortality, and local residents oppose huge visible infrastructure. The planning of power lines has fittingly been called a ''field of tension''~\cite{bevanger2011optimal}.

\begin{figure}[ht]
    \centering
    \includegraphics[width=0.8\textwidth]{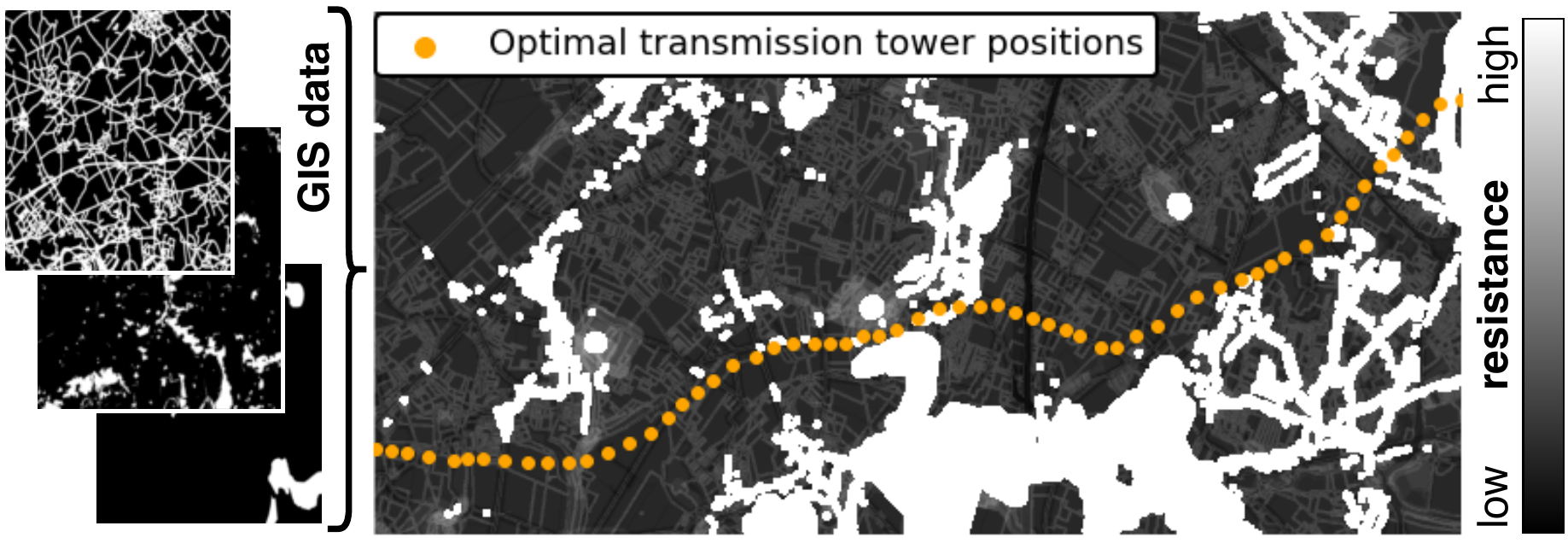}
    \caption{An optimal transmission line planned with our framework. The resistances given with GIS-data are minimized together with the angles along the route, yielding the globally optimal pylon placement.}
    \label{fig:overview}
\end{figure}

Nevertheless, power infrastructure planning is oftentimes highly manual, mostly regarded as a communicative process between experts~\cite{bevanger2011optimal, grossardt2001analytic}. In contrast, the optimal placement of power lines can be formulated as a simple combinatorial problem, when provided with appropriate data extracted from geographic information systems (GIS). Possible computational approaches for optimal layout of power lines have been discussed for a long time in the literature. In most cases the area is rasterized and each cell is assigned a cost or resistance according to GIS data. 
For example, \citet{popp1963electronic} presented very early work on optimal tower placement and even the selection of pylon heights based on the line sag and terrain. 
%
More recently, \citet{bevanger2011optimal} proposed an environment-friendly power-line routing in Norway. A least-cost path (LCP) toolbox was published in the context of the same project~\cite{hanssen2018spatial}. Similarly, \citet{monteiro2005gis} and \citet{de2016least} describe planning algorithms that consider the slope of the terrain and costs on direction change~\cite{monteiro2005gis}. Others develop decision support software that compute not only a single route, but multiple (pareto-optimal) alternatives, using multi-objective optimization techniques~\cite{bachmann2018multi} or valley-finding algorithms applied on the cost surface~\cite{schito2021determining}.

Despite the advances on optimizing power infrastructure layout, the approaches above suffer from a common drawback, namely that they only compute a shortest \textit{route} through a regular grid, where each cell is connected to its 8-neighborhood. 
Transmission towers are only placed as a second step, yielding sub-optimal solutions: Crucially, it is disregarded that the cost\footnote{The term ''resistance'' is used interchangeably with ''cost'' here and thus does not refer to electrical resistance. Instead, it denotes the summarized resistance of several geographic layers.} of placing a power tower in a cell can differ significantly from the cost of traversing a cell with a cable.  We instead suggest a graph model to minimize the resistance of the \textit{pylon} placement, offering several advantages as detailed in \autoref{fig:table}.
%
\\ 
An additional challenge in power infrastructure planning is the avoidance of angles, since angles increase construction costs and electrical resistance. 
Work on angle-constrained shortest path computation in a graph model similar to ours is given by \citet{santos2019optimizing}. 
Building up on~\citet{piveteau2017novel}, they construct a \textit{line graph} in a pre-computed corridor to model angle resistance. These works are important steps towards optimized pylon spotting, but leave room for improvement with respect to computational efficiency.
%
%
%
%

\textbf{Contribution.}
Here, we aim to fill this gap with an efficient framework for the globally optimal placement of pylons in power transmission lines, offering 
the following functionality:
\begin{itemize}
    \vspace{-1em}
    \item Graph modelling of transmission line planning with given (possibly different) cable and pylon resistances.
    \item Provably efficient algorithms to compute the \textit{minimal-angle} shortest path. The baseline approach, a line graph~\cite{santos2019optimizing, piveteau2017novel}, has size $\mathcal{O}(\sum_{v\in V} \delta_v^- \cdot \delta_v^+)$ for a graph with in-degree $\delta_v^-$ and out-degree $\delta_v^+$ at vertex $v\in V$. The Bellman Ford (BF) algorithm thus requires $\mathcal{O}(|V|\cdot \sum_{v\in V} \delta_v^- \cdot \delta_v^+)$ operations. In contrast, we adapt BF to operate on the normal graph with size $\mathcal{O}(\sum_{v\in V} \delta_v^+)$, and reduce the runtime to $\mathcal{O}\big(|V| \cdot \sum_v (\delta_v^- + \delta_v^+) \cdot \log (\delta_v^- \delta_v^+)\big)$ with a convex angle cost function.
    \item An iterative procedure to achieve time and space feasibility with arbitrarily large project regions
    \item Methods to compute a set of $k$ paths that are short but also \textit{diverse} among each other according to a suitable metric. This is beneficial for planners as a decision support tool, to compare and analyse alternatives in different regions.
\end{itemize}

We evaluate our methods on realistic power infrastructure projects in collaboration with \textit{Gilytics}, a software company working on visualization and optimization software for linear infrastructure planning.
We release an open-source package called LION (LInear Optimization Networks) providing all main algorithms. An exemplary output is shown in \autoref{fig:overview}.

\FloatBarrier

%% file: chapter/02setup.tex
\section{Problem setting}

The input to our framework is rasterized geographic data, here with a resolution of 10x10 meter cells. The ''resistance raster'' $C$, describing the costs per cell, is constructed from binary layers $L_1, \dots, L_r$ (\autoref{fig:overview} left) and corresponding weights $w_1, \dots, w_r$. Each layer indicates the occurrence of a geographic feature, e.g. 
$L_i[x,y]=1$ could indicate that cell $(x,y)$ is part of a forest. A feature $L_j$ that prohibits the placement of a transmission tower is integrated with $w_j=\infty$. We refer to previous works for approaches to decide how to weight and summarize these layers~\cite{bevanger2011optimal, 
grossardt2001analytic}. 
Here, we assume that the resistance raster $C$ is simply a weighted sum of the individual layers: 
\begin{gather}
    C =\displaystyle \sum_{i=1}^{r} w_i \cdot L_i
\end{gather}

\subsection{Graph model}\label{sec:graphbuild}
Similar to \cite{santos2019optimizing, piveteau2017novel}, we model power infrastructure planning with a graph $G=(V,E)$ with $n:=|V|$ vertices representing transmission towers, and $m:=|E|$ directed or undirected edges, corresponding to the cable between towers. \autoref{fig:lion_graph} visualizes the graph layout: Each vertex is one cell in the raster, i.e. there exists a bijective function $l: V \rightarrow \mathbb{N} \times \mathbb{N}$, such that $l(v)$ outputs the set of raster coordinates modelled by vertex $v$. Conversely, $l^{-1}(x,y)$ is the vertex representing the cell at coordinates $(x,y)$. When not noted otherwise, $s\in V$ will denote the source vertex and $t\in V$ the target. Furthermore, there is an edge connecting two vertices $u,v$ if and only if their respective pylons could be neighbors in a transmission line, which is the case if their distance is between two bounds $d_{min}$ and $d_{max}$. 
\\
With simplifying assumptions, the setting can be modeled in a directed acyclic graph (DAG) which significantly accelerates shortest path algorithms. 
Let $\vec{a}_{u,v} := l(u)-l(v)$ be the vector from the cell of vertex $u$ to cell of vertex $v$. We then place an edge between a pair of vertices if their distance is between $d_{min}$ and $d_{max}$, but additionally require that $\vec{a}_{u,v}$ diverges less than a certain maximum angle $\theta_{\alpha}$ from the straight line connection from source $s$ to target $t$:
\begin{equation}
\begin{gathered}
E = \Big\{(u,v) \ |\ d_{min}\leq \Vert \vec{a}_{u,v} \Vert \leq d_{max} \land\ cos\big(\frac{\vec{a}_{u,v}\cdot \vec{a}_{s,t}}{\Vert \vec{a}_{u,v}\Vert \Vert \vec{a}_{s,t}\Vert}\big)<\theta_{\alpha}\Big\}
\end{gathered}
\end{equation}
This assumption is justified since in power infrastructure layout it is very unlikely that a tower is placed ''behind'' another one in a transmission line from $s$ to $t$. If $\theta_{\alpha}<\frac{\pi}{2}$, then the graph is a DAG, because the neighbors of each vertex are the cells on the semicircular ring directed towards the target (\autoref{fig:lion_graph}). Note however that the algorithms presented in the following sections are not restricted to DAGs.
\input{chapter/big_figure}

\subsection{From resistances to edge costs}\label{sec:edgecost}
%
In contrast to two-stage \textit{line routing} approaches (\autoref{fig:prior}) used in previous work~\cite{monteiro2005gis, vega1996image, west1997terrain, li2013optimization}, a \textit{pylon-spotting} approach allows a) to differentiate pylon and cable costs, b) to find a globally optimal pylon spotting and c) to minimize the angles at the pylons. In \autoref{fig:table} both approaches are contrasted in more detail.
Pylon and cable resistances can be computed independently with different weights, such that $C_p[x,y]$ is the cost of placing a pylon in cell $(x,y)$ and $C_c[x,y]$ denotes the cost to traverse $(x,y)$ with a cable.
In order to translate $C_p$ and $C_c$ into edge costs in a graph, the cable cost must be summed over several cells, since the cable resembles a 
straight line above a discrete raster of cells (black arrow in \autoref{fig:lion_graph}). To approximate this line, we utilize a well-known line drawing algorithm from computer graphics, the Bresenham line~\cite{bresenham1965algorithm}. 
The computation of the Bresenham line $M_e$ for an edge $e=(u,v)$ yields a set of coordinates corresponding to the raster cells on the line, formally $M_e=\{l(u), (x_2,y_2), \dots, (x_{d-1}, y_{d-1}), l(v)\}$. 
Pylon and cable costs are then combined with a weight $w_c$ that describes the importance of the cable cost (i.e. the average resistance of the cells in $M_e$) in contrast to the pylon cost: 
\begin{gather}\label{eq:edgecost}
\displaystyle c(e) = \frac{C_p[l(u)] + C_p[l(v)]}{2} + w_c\cdot\frac{1}{d} \sum_{(x,y)\in M_e} C_c[x,y]    
\end{gather}

%% file: chapter/big_figure.tex
\begin{figure*}[ht]
\centering
    \begin{subfigure}[t]{0.27\textwidth}
        \centering
        \includegraphics[width=\textwidth]{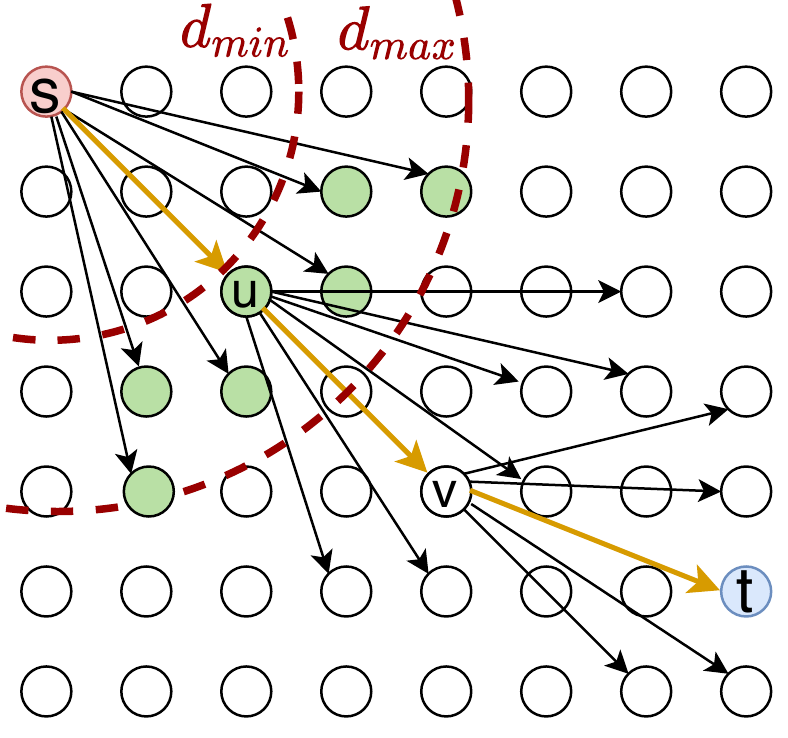}
        \caption{Pylon spotting (ours)}
        \label{fig:lion_graph}
    \end{subfigure}
    \begin{subfigure}[t]{0.37\textwidth}
            \centering
            \includegraphics[width=\textwidth]{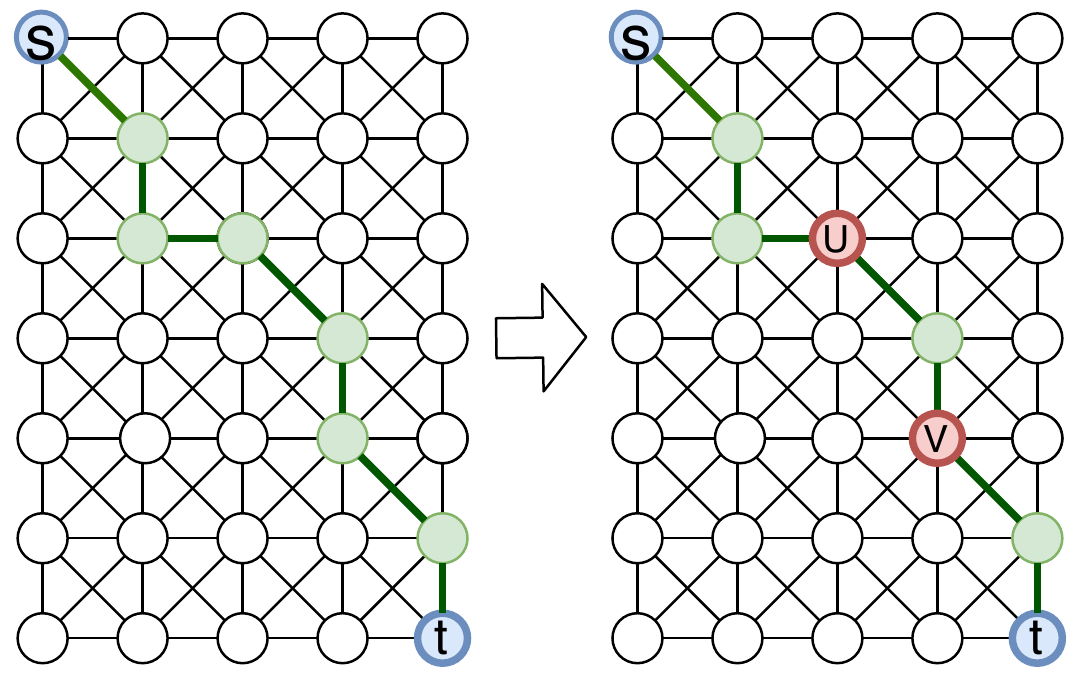}
    \caption{Line routing (two-stage approach often used in prior work) \label{fig:prior}}
    \end{subfigure}
    \begin{subtable}[t]{0.33\textwidth}
        \resizebox{\textwidth}{!}{
        \begin{tabular}[b]{c|c}\hline
         \textbf{Pylon spotting (ours)} & \textbf{Line routing (prior work)} \\ \hline
          \begin{tabular}[c]{@{}c@{}}  Global optimization\\ of pylon positions \end{tabular} & \begin{tabular}[c]{@{}c@{}} Two stage approach \end{tabular}
             \\ \hline
          \begin{tabular}[c]{@{}c@{}} Forbidden regions can \\ be traversed by cable \end{tabular}
            & 
            \begin{tabular}[c]{@{}c@{}}  Forbidden areas \\ cannot be crossed \end{tabular}
             \\ \hline
            \begin{tabular}[c]{@{}c@{}}  Allows to differentiate \\ resistances for \\ pylons and cable \end{tabular}
             & 
             \begin{tabular}[c]{@{}c@{}}  No differentiation \\ possible \end{tabular}
              \\ \hline
            \begin{tabular}[c]{@{}c@{}}  Will find feasible \\ pylon spotting \\ (if existent) \end{tabular}
           & 
            \begin{tabular}[c]{@{}c@{}}  Placing pylons along \\ the precomputed route \\ might be impossible \end{tabular} \\ \hline
            \begin{tabular}[c]{@{}c@{}}  Direct angle \\ minimization \end{tabular}
           & 
            \begin{tabular}[c]{@{}c@{}}  Angles along route \\ $\neq$ angles at pylons \end{tabular}
            \\ \hline
        \end{tabular}}
        \caption{Advantages of our approach \label{fig:table}}
    \end{subtable}
\caption{Comparison of our graph model (\subref{fig:lion_graph}) to previous work (\subref{fig:prior}). In  \subref{fig:lion_graph}) each raster cell is a vertex, and edges are placed between each pair of vertices with distance between $d_{min}$ and $d_{max}$. Only edges with sufficiently low angle to the straight line from $s$ to $t$ are allowed, thereby yielding a directed acyclic graph (DAG). In previous approaches (\subref{fig:prior}), first the optimal \textit{route} through a regular grid is found, and only in the second step transmission towers are placed along the route. The drawbacks of this line-routing compared to direct pylon spotting are listed in \subref{fig:table}).
}
\label{fig:graphbuild}
\end{figure*}

%% file: chapter/03_angle.tex
\section{Shortest path algorithms for the minimal-angle least-resistance power line}
We use the Bellman-Ford (BF)~\cite{bellman1958routing, ford1956network} algorithm to compute the shortest (minimal resistance) path, since it is more space-efficient than Dijkstra and the runtime can be reduced if an upper bound on the number of vertices on the shortest path is known. While the BF algorithm will minimize cable and pylon resistances according to \autoref{eq:edgecost}, it is non-trivial to regard di-edge costs such as \textit{angle} costs. The latter problem was formalized as a special case of the \textit{Quadratic Shortest Path Problem (QSPP)}\cite{rostami2015quadratic}, named the Adjacent QSPP. In AQSPP the objective is to minimize an interaction cost of adjacent edges along with the normal edge costs. Even though AQSPP was shown to be NP-hard when the graph is not acyclic~\cite{rostami2018quadratic, hu2018special}, an auxiliary graph structure, namely the \textit{line graph} or \textit{dual graph}, has been exploited in applications to deal with di-edge costs, including power line ~\cite{santos2019optimizing, piveteau2017novel} and other infrastructure planning~\cite{caldwell1961finding, winter2002modeling}. 

However, a major drawback of the line graph is its size: A directed graph with $\delta_v^-$ incoming edges and $\delta_v^+$ outgoing edges at vertex $v$ results in $\sum_v \delta_v^- \cdot \delta_v^+$ 
edges in the line graph. Less memory-consuming representations have been proposed in work on turn constraints in road networks \cite{delling2011customizable}, where Dijkstra is modified to consider turn-tables associated with each vertex. 
Here, we propose a 
variant of the Bellman-Ford algorithm named MABF, that computes the minimal-angle \textit{walk} (or path under certain conditions) in the original graph. 

\FloatBarrier
\subsection{Minimal-angle Bellman Ford algorithm (MABF)}
In addition to the edge weight given by the cost function $c(e): E \longrightarrow \mathbb{R}$ in \autoref{eq:edgecost}, let $\alpha: [0,180] \longrightarrow \mathbb{R}^{+}$ be an angle cost function. We aim to compute a distance map $D$ and a predecessor map $T$ from which the shortest path can be reconstructed. Specifically, $D[e]$ describes the accumulated costs for the (current) shortest path from the source $s$ to the edge $e$, including angle costs. $D$ is initialized to infinity for all edges, except for the outgoing edges $e=(s,-)$ of the source, which are instead set to their respective edge cost, formally $D[e] = c(e)\ \ \forall e=(s, -)$.  

\begin{figure}[ht]
    \centering
    \begin{subfigure}[b]{0.35\columnwidth}
    \centering
    \includegraphics[width=\textwidth]{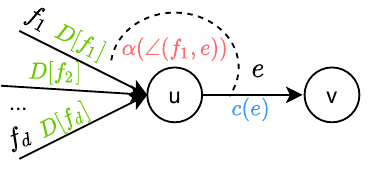}
    \caption{Update step in minimal-angle BF}
    \label{fig:bfangle}
    \end{subfigure}
    \hspace{4em}
    \begin{subfigure}[b]{0.25\columnwidth}
            \centering
            \includegraphics[width=\textwidth]{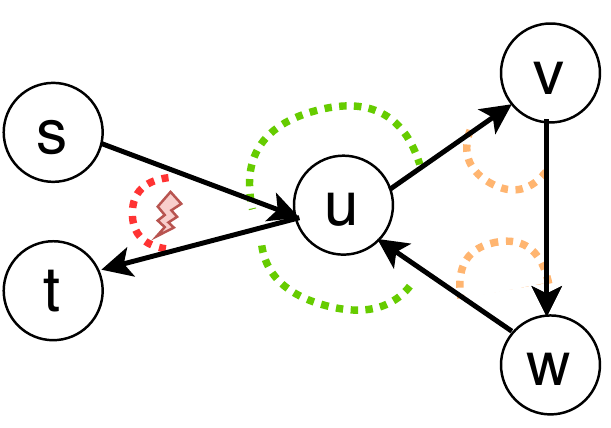}
    \caption{Cycle on optimal \textit{walk} \label{fig:cycle}}
    \end{subfigure}
    
\caption{Minimal-angle shortest path computation: As shown in \subref{fig:bfangle}), the minimum over the incoming edges $f_1, \dots, f_d$ must be taken to find the optimal predecessor of edge $e$ with respect to the distance $D$ and angle cost $\alpha$. For non-concave $\alpha$, the algorithm can yield cyclic paths as in \subref{fig:cycle}), since high angle costs (red) can be avoided by using more but smaller angles along the cycle (green and orange). 
}
\label{fig:detourcycle}
\end{figure}

Algorithm \ref{alg:alg1} outlines the proposed procedure to find the angle-penalized least cost \textit{walk} from $s$ to $t$. Note that in the general case the route can contain cycles because thereby sharp turns (high angle costs) are avoided (see \autoref{fig:cycle}). We first describe the algorithm and then show that the result will be a \textit{path} when $\alpha$ is concave.
We assume an upper bound $p$ on the number of edges in the shortest path. Note that $p<n$, but due to the raster-layout of the power infrastructure graph, it is usually safe to set $p$ to a value much smaller than $n$. BF computes the shortest path tree by updating the distances to all edges for $p$ iterations. 
\autoref{fig:bfangle} visualizes the new update step applied to edge $e=(u,v)$: Its distance from $s$ is given by the minimum over the distances of its possibly preceding edges, together with the respective angle cost:
\begin{gather}\label{eq:update}
D[e]=\displaystyle\min_{f\in E,\ f=(-, u)} D[f] + c[e] + \alpha(\sphericalangle( f, e))
\end{gather}

\begin{algorithm}[ht]
\Indm 
\SetAlgoLined
\Indp{Input: $G$, $s$, $c$, $\alpha$}\\
 $D[e] = \infty\ \ \forall e\in E$\;
 $D[(s, -)] = c((s, -))$\;
 $T[e] = e\ \ \forall e\in E$\;
 \For{$p$ iterations}{
 \For{$i\gets0$ \KwTo $m$}{
 \SetInd{0.1em}{0em}
  $(u, v) \leftarrow e_i$ \;
  $D[e_i] = \displaystyle\min_{f\in E,\ f=(-,u)} D[f] + c[e_i] + \alpha(\sphericalangle(f, e_i)) $\;
  $T[e_i] = \displaystyle\argmin_{f\in E,\ f=(-,u)} D[f] + c[e_i] + \alpha(\sphericalangle (f, e_i)) $
 }
 }
 \caption{Minimal-angle Bellman-Ford algorithm}
 \label{alg:alg1}
\end{algorithm}

In the following, it is shown that this variant of the Bellman-Ford algorithm computes the correct distances.

\begin{theorem}\label{theorem:bf}
Let $e_t = (-, t)$ be any edge in $G$. Furthermore, let $P^*$ be a least cost walk of length $|P^*| \leq p$ from $s$ to $t$, and denote the edges on $P^*$ as $P^*=\{e_1, ..., e_t\}$.

After $p$ update iterations in Algorithm~\ref{alg:alg1}, it holds that $$D[e_t] = \displaystyle c(e_1) + \sum_{i=2}^{t} c(e_i) + \alpha(\sphericalangle (e_{i-1}, e_i)) := c(P^*)$$ and $c(P^*)$ is the cost of the shortest minimal-angle walk. 
\end{theorem}
\begin{proof}
Proof by induction on $p$. The base case $p=1$ is given by the initialization $D[(s,-)] = c((s,-))$.
As the induction step, consider the least cost walk $P^*$ from $s$ to $e_t=(u, t)$ such that ($c(P^*)$ is minimal. Let $f^*$ be the edge preceding $e_t$ on $P^*$, so $f^*$ is of the form $(-,u)$. Then $P^*\setminus e_t$ is a shortest walk from $s$ to $f^*$ of length $p-1$. By induction, after $p-1$ iterations we have $D[f^*]=c(P^*\setminus e_t)$. More general, for all incoming edges of $u$, defined as $f_j: f_j\in E, f_j=(-,u)$, $D[f_j]$ contains the total cost of the shortest walk $P_j$ from $s$ to $f_j$ satisfying $|P_j|\leq p-1$. Then the distance of $e_t$ is updated according to \autoref{eq:update}:
\begin{gather*}
    D[e_t] = \min_j D[f_j] + c[e_t] + \alpha(\sphericalangle(f_j, e_t)) \\ 
    = D[f^*] + c[e_t] + \alpha(\sphericalangle(f^*, e_t)) \\
    = c(P^*\setminus e_t) + c[e_t] + \alpha(\sphericalangle(f^*, e_t))
    \\ = c(P^*)
\end{gather*}
Thus, $f^*$ is correctly assigned as the predecessor of $e_t$.
\end{proof}

\subsection{Computing acyclic paths}
As shown in \autoref{fig:cycle}, the route with least angle resistances might contain cycles, since strong angles (red) are avoided with multiple small angles (green, orange). The following lemma proves a sufficient condition that guarantees that the output is a \textit{path}.
\begin{lemma}
If the angle cost function $\alpha: [0,180]\longrightarrow \mathbb{R}^+$ is concave, algorithm~\ref{alg:alg1} returns a shortest path.
\end{lemma}
\begin{proof}
According to the Closed-Path theorem formulated by \citet{abelson1986turtle}, ''a total turning along any closed path is an integer multiple of $360^{\circ}$'', which applies to a cycle in a digraph. However, the angle at the entry vertex is not included in our case, leaving at least $180^{\circ}$ for the cycle. Let $\sigma_1, \dots, \sigma_n$ denote all angles encountered on the cycle. Then $\forall i:\ \sigma_i\in [0,180] \text{ and } \sum_{i=1}^n \sigma_i \geq 180$ due to the Closed-Path theorem. It is further well-known that for a concave function $\alpha$ it holds that $\sum_i \alpha(x_i) \geq \alpha(\sum_i x_i)$, and therefore it follows $$\sum_{i=1}^n \alpha(\sigma_i) \geq \alpha\Big(\sum_{i=1}^n \sigma_i \Big) \geq \alpha(180)\ .$$ We conclude that the angle costs accumulating on the cycle will be at least as large as the angle cost that would be part of the direct path, which is at most $\alpha(180^{\circ})$.
\end{proof}

\subsection{Time complexity}
The above results show the correctness of algorithm~\ref{alg:alg1} to compute the minimal-angle shortest walk for general and the shortest path for concave angle cost functions. While it improves on the space requirements in contrast to a line graph construction, namely $\mathcal{O}(m)$ in contrast to $\mathcal{O}(\sum_v \delta_v^- \cdot \delta_v^+)$, it has the same runtime: In each update of an edge, the minimum over the incoming edges is taken, leading to a runtime of $\mathcal{O}\big(p \cdot \sum_v \delta_v^- \cdot \delta_v^+\big)$, corresponding to the runtime of normal BF applied on the line graph. We however claim that the runtime can be improved significantly, if the angle cost function is a discrete (step) function or convex.

\begin{lemma}\label{lemma:runtime1}
If $\alpha: [0,180]\longrightarrow S$ is a step function with $|S|$ possible angle costs ($|S|<\delta_v^-, \delta_v^+$), the shortest walk can be computed in time $\mathcal{O}\big(p \cdot \sum_v (|S|\cdot \delta_v^- + \delta_v^+) \cdot \log (|S| \delta_v^- \delta_v^+)\big)$.
\end{lemma}
\begin{lemma}\label{lemma:runtime2}
For convex, monotonically increasing $\alpha$ the shortest walk can be computed in time $\mathcal{O}\big(p \cdot \sum_v (\delta_v^- + \delta_v^+) \cdot \log (\delta_v^- \delta_v^+)\big)$.
\end{lemma}

The proofs for Lemma~\ref{lemma:runtime1} and Lemma~\ref{lemma:runtime2} are found in Appendix~\ref{discrete} and ~\ref{convex} respectively. Both algorithms are implemented in the LION package and are used automatically based on the size of the instance and the selected angle cost function. To summarize, we provide space- and time efficient shortest path algorithms that minimize angles along the path and omit the need for a line graph construction. 

%% file: chapter/035_pipeline.tex
\section{Improving space complexity with an iterative multi-scale approach}\label{sec:multiscale}
Despite space-efficient algorithms, the graphs can still become extremely large, in the order of billions of edges. In particular, $m$ increases \textit{quartic} with the resolution: Not only the number of vertices grows quadratically, but also the number of neighbors per vertex. We take advantage of this observation with a multi-scale approach that narrows down the region of interest iteratively. \\Initially, the path is computed on a coarse downsampled version of the instance. Then the resolution is increased over several steps while simultaneously restricting the project region to a corridor around the previously computed path. Let $(r_1, \dots, r_q)$ be a list of downsampling factors, and let $\vartheta$ be an upper bound on the number of edges (i.e. a memory limit).
\begin{figure}[b!]
    \centering
    \includegraphics[width=0.8\textwidth]{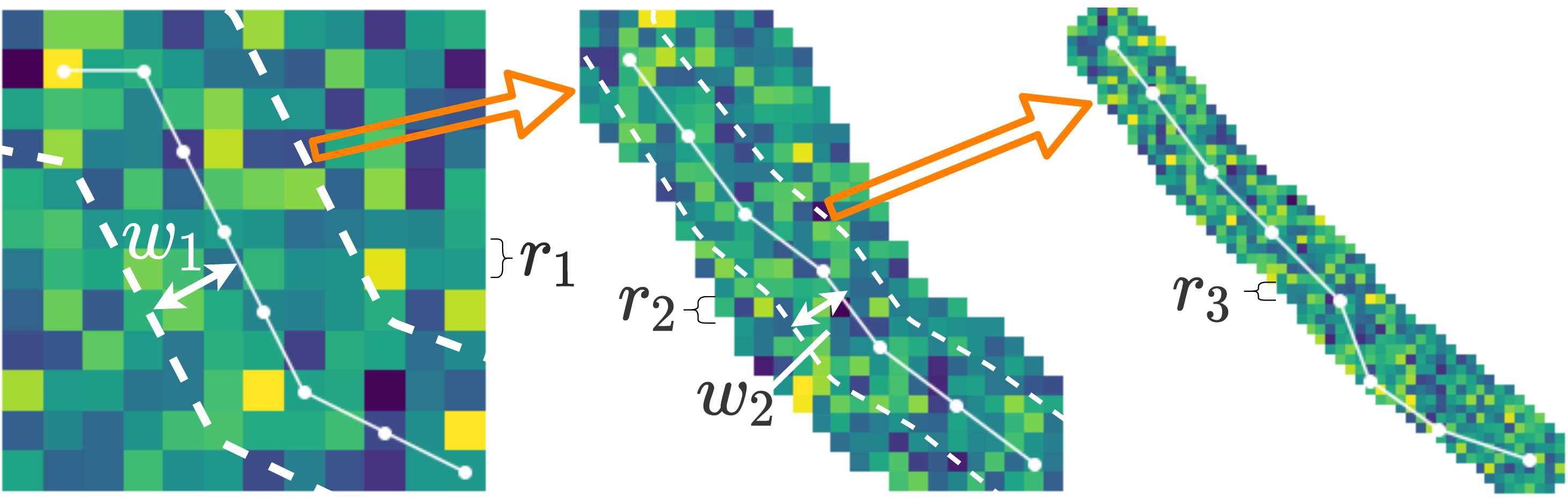}
    \caption{Iterative multi-scale approach: The instance is downsampled by $r_1$ and the optimal path $P$ is computed (left). For the next iteration, the region is restricted to a corridor of width $w_1$ around $P$ (dotted line left), and again the best path in the corridor with resolution $r_2$ is yielded. The output is the best path through the final region (resolution $r_3$ in corridor with size $w_2$).}
    \label{fig:pipeline}
\end{figure}
\autoref{fig:pipeline} provides a visual explanation: The resolution of the instance is decreased by a factor of $r_1$ and the optimal path $P_1$ is found at this scale. For the next iteration, the project region is restricted to a corridor around $P_2$ (white dotted line). In this corridor, the optimal path is computed with resolution corresponding to $r_2$. These steps are repeated until reaching the highest resolution (usually $r_q=1$). In each step $i$, the width $w_i$ of the corridor is set to the maximum feasible width such that the number of edges remains below $\vartheta$ in the next iteration (where the instance is downsampled by $r_{i+1}$). Importantly, this multi-scale approach can also be combined with the computation of $k$ diverse shortest paths explained in the following section: The corridor for the next iteration is in this case a buffer around all $k$ paths. In general, the multi-scale procedure will yield only a \textit{locally} optimal path, but with a sufficiently homogeneous instance, good results can be achieved in significantly lower time and space.

%% file: chapter/04_ksp.tex
\section{Trading off soft and hard criteria with k diverse shortest paths}\label{ksp}

Since the single optimal path is solely based on the geographic input data (hard criteria), it might not comply with soft criteria coming up in the analysis of experienced planners. To trade off such criteria in discussions between planners and involved stakeholders, it is of interest to propose a set of multiple short paths for the sake of comparison. \\ 
A very efficient method to compute $k$ shortest paths is given with Eppstein's algorithm~\cite{eppstein1998finding} if paths are allowed to be loopy.
Here, it is first shown that Eppstein's algorithm can be easily adapted to return the $k$ minimal-angle shortest paths. 
As defined above, let $p$ be an upper bound on the path length.

\begin{lemma}
The $k$ shortest paths can be computed in $\mathcal{O}(p \cdot \sum_v \delta_v^- \cdot \delta_v^+ + m \log m + kp)$
\end{lemma}
\begin{proof}
The proof is given by the following adaptation of Eppstein's algorithm (EA): As in EA, not only the distance $D_s$ and predecessor map $T_s$ for the shortest path tree rooted in $s$, but also the tree rooted in $t$ is computed with reversed edge directions, yielding $D_t$, $T_t$. With algorithm~\ref{alg:alg1} this takes $\mathcal{O}\big(p \cdot \sum_v \delta_v^- \cdot \delta_v^+\big)$ operations. 
Following EA, $D_s $ and $D_t$ are summed up point-wise and the edge costs $c(e)$ are subtracted since they appear twice. The resulting map $S$, with $S[e] = D_s[e] + D_t[e] - c(e)$, now describes the distance of the shortest path through $e$. Then $S$ is sorted $\mathcal{O}(m \log(m))$, and the paths are reconstructed for the edges with the $k$ lowest values in $S$, which can be done in $\mathcal{O}(kp)$ using $T_s$ and $T_t$.
\end{proof}

Enumerating the $k$ shortest paths often leads to highly similar paths. 
Thus, the problem of \textit{diverse} path finding has been introduced by \citet{liu2017finding} as the ''Top-$k$ Shortest Paths with Diversity (KSPD)'', which has been shown to be NP-hard~\cite{chondrogiannis2018finding}. Both \citet{liu2017finding} and \citet{chondrogiannis2015alternative} provide greedy methods to find $k$ paths that are sufficiently diverse with respect to a threshold $\theta$ of a suitable path diversity metric. \citet{akgun2000finding} on the other hand consider the reversed formulation, namely to find the $k$ most diverse paths with a threshold on the path costs, thereby turning the problem into a\textit{ p-dispersion} problem. 
Here, we were first interested to understand for which metrics it is even possible to find \textit{one} diverse path efficiently.

\begin{definition}\label{def:diverseselection}
Given a graph $G=(V,E)$, a path $P$ and a metric $M$, we define the \textbf{diverse path selection problem} as the problem to find an alternative path $\hat{P}$ where the distance exceeds a similarity threshold $\theta$, i.e. $d_M(P, \hat{P})>\theta$.
\end{definition}

We first consider the Yau-Hausdorff distance $d_{Y}$, which is the maximum Euclidean distance between two sets $P$, $Q$, and was shown to be a metric by \citet{tian2015two}.
\begin{gather}
d_{Y}(P,Q) = \max\{\max_{p\in P} \min_{q\in Q} \Vert p-q\Vert, \max_{q\in Q} \min_{p\in P} \Vert p-q\Vert\}
\end{gather}

\begin{lemma}\label{lemma:3}
Diverse path selection according to the Yau-Hausdorff distance is decidable in polynomial time.
\end{lemma}
\begin{proof}
After computing both shortest path trees, one pass over all edges is required to determine which vertices have sufficient Euclidean distance to the previous path(s) $P_{prev}$. Note that in order to solve the diverse path selection problem, \textit{any} diverse path is sufficient, but with Eppstein's algorithm we can also find the \textit{shortest} diverse path by finding the sufficiently diverse path with minimal $S[e]$:
\begin{gather*}
    \displaystyle e^* = \argmin_{e=(u,v):\ \min_{w\in P_{prev}} \Vert u - w \Vert > \theta} S[e].
\end{gather*}
If $e^* =(u^*,v^*)$ exists and has non-infinity path distance, the shortest path $P^*$ through $e^*$ can be computed from $T_s$ and $T_t$. 
The Yau-Hausdorff distance of $P^*$ to $P_{prev}$ is the maximum distance of all vertices of $P^*$, and thus at least as large as the distance of $u^*$ to $P$:
\begin{gather*}
    d_{Y}(P^*,P_{prev}) \geq \max_{p\in P^*} \min_{q\in P_{prev}} \Vert p-q\Vert \geq \min_{q\in P_{prev}} \Vert u^*-q\Vert
\end{gather*}
Conversely, if no such edge exists then clearly there is no path with sufficient Yau-Hausdorff distance, because the maximum distance must be attained at some vertex.
\end{proof}

Lemma~\ref{lemma:3} thus suggests an efficient \textit{greedy} algorithm to compute $k$ diverse shortest paths with respect to $d_Y$. The runtime is given by computing the shortest path trees and subsequently computing the path distance ($\mathcal{O}(kp)$) for all $n$ vertices for $k$ times, resulting in  $\mathcal{O}(mpd + k^2np)$.\\
On the other hand, despite the proposed methods and promising computational results in \cite{chondrogiannis2018finding,liu2017finding}, we 
find that diverse path selection is NP-hard for the Jaccard distance metric $d_J$ and the mean Euclidean distance $d_M$, 
\begin{gather}
    \displaystyle d_M(P,Q) = \max\{\frac{1}{|P|}\sum_{p\in P} \min_{q\in Q} \Vert p-q\Vert, \frac{1}{|Q|}\sum_{q\in Q} \min_{p\in P} \Vert p-q\Vert\}
\end{gather}

\begin{lemma}
With the Jaccard distance or the average Euclidean distance metric, the diverse path selection problem is NP-hard.
\end{lemma}
\begin{proof}
Given a graph $G=(V,E)$ and $n=|V|$, we construct an auxiliary graph $H=(V^{'}, E^{'})$ as in \autoref{fig:eucldisthard}: 
\begin{gather*}
V^{'} = V \cup \{v^{'}_i\ |\ i\in [1..n]\} \\
E^{'} = E\cup \Big\{(v^{'}_i, v^{'}_{i+1})\ |\ i\in [1..n-1]\Big\}
\\ \cup
\Big\{(v_i, v^{'}_{\big\lfloor\frac{n}{2}\big\rfloor}) \ |\ \forall i\in [1..n]\Big\}
\cup \Big\{(v_i, v^{'}_{\big\lceil\frac{n+1}{2}\big\rceil}) \ |\ \forall i\in [1..n]\Big\}
\end{gather*}

$G$ can be any graph, but here we arrange all its vertices in a straight line in an Euclidean space, and place $v^{'}_1\dots v^{'}_{n}$ in a line parallel to $v_1\dots v_{n}$ such that $d_M(v_i, v^{'}_i) = 2\ \ \forall i$.

We claim that the Hamiltonian problem can be decided for $G$ if it can be decided whether there exists a diverse alternative to $P = (v^{'}_1, v^{'}_2, \dots, v^{'}_n)$ (horizontal green line in \autoref{fig:eucldisthard}). The threshold for the minimum Jaccard distance is set to $\theta = 0.5$. Observe that all vertices of $P$ will also be part of the alternative path, and only the edges from $v^{'}_{\big\lceil\frac{n+1}{2}\big\rceil}$ and $v^{'}_{\big\lfloor\frac{n}{2}\big\rfloor}$ allow the path to traverse other vertices than $P$. The Jaccard distance corresponds to the inverted IoU, so $\theta = 0.5$ requires the path to use at least $|P|$ new vertices. But since $|P| = n$, the only possible alternative is a path that visits all vertices of $G$, i.e. a Hamiltonian path in $G$. If no alternative is found, it is implied that there is no Hamilton path because the connections from $v^{'}_{\big\lceil\frac{n+1}{2}\big\rceil}$ and $v^{'}_{\big\lfloor\frac{n}{2}\big\rfloor}$ to \textit{all} vertices of $G$ ensure that it would be found if existent.

Similarly, with respect to the mean-Eucledian distance $d_M$ we set $\theta = 1$. 
If $k$ vertices of $G$ are used ($k \leq n$), then 
$$d(P, P^*) = \frac{n \cdot 0 + k\cdot 2}{n+k} \implies (d(P, P^*) =1 \iff k=n)$$ 
Thus, a diverse alternative exists iff a Hamiltonian path exists.
\begin{figure}
    \centering
    \includegraphics[width=0.8\textwidth]{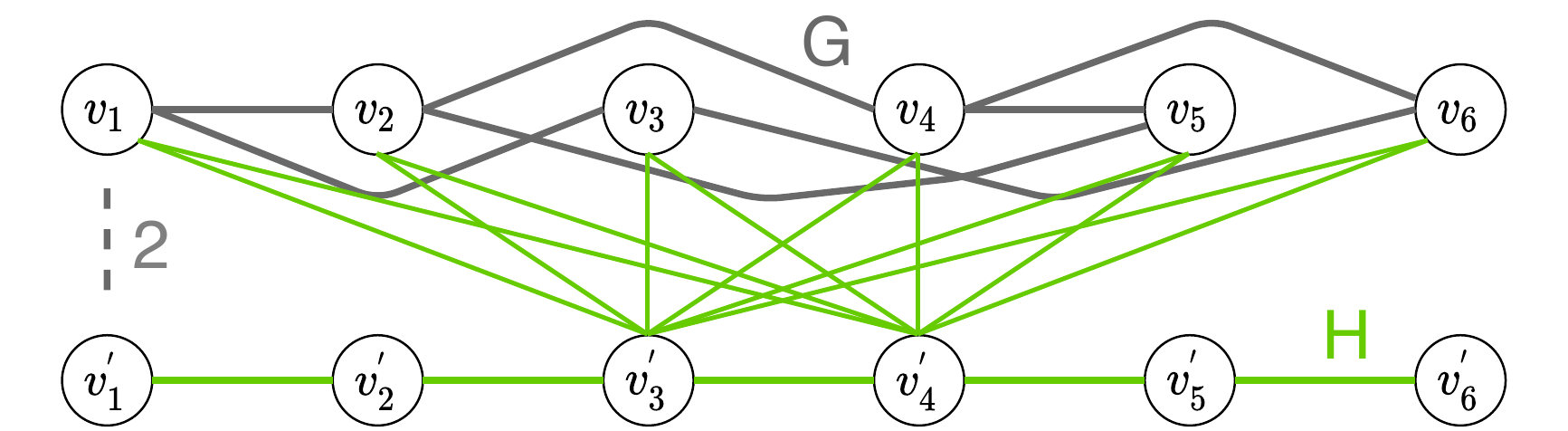}
    \caption{NP-hardness of finding the shortest path with a lower bound on the mean Euclidean distance or the Jaccard distance. If a diverse path exists, G contains a Hamilton path.}
    \label{fig:eucldisthard}
\end{figure}
\end{proof}
Nevertheless, one can still utilize Eppstein's algorithm to enumerate paths sorted by distance, and then select the ones greedily that have sufficient average Euclidean or Jaccard distance to the previously chosen paths, similarly to the FindKSP algorithm by \citet{liu2017finding}. We therefore implemented and compared five computational methods:
\begin{itemize}
    \item \textbf{find-ksp-max}: greedy path selection according to $d_Y$, guaranteed to give the correct result in polynomial time
    \item \textbf{find-ksp-mean}: greedy path selection according to $d_M$
    \item \textbf{greedy-set}: greedy path selection according to $d_J$
    \item \textbf{k-dispersion}: 2-approximation~\cite{ravi1991facility} for p-dispersion problem, i.e. find most diverse set of paths given a threshold on the cost (similar to \cite{akgun2000finding})
    \item \textbf{corridor-penalizing}: Soft variant of \textbf{find-ksp-max}, where the vertices that are close to the previous paths are only penalized instead of excluded
\end{itemize}

Examples of five diverse minimal-angle and least resistance paths, computed with \textbf{find-ksp-max}, are shown in \autoref{fig:ksp_example}.

%% file: chapter/05_experiments.tex
\section{Computational results}

In collaboration with the software company Gilytics we evaluate our methods on two transmission line planning projects, labeled \textbf{instance~1} and \textbf{instance~2} in the following. \textbf{Instance~1} is an exemplary project in Austria with GIS data publicly available\footnote{\url{http://www.noe.gv.at/noe/Karten-Geoinformationen/DownloadGeodatenKarten.html}} (resistance raster shown in \autoref{fig:ksp_example}). \textbf{Instance~2} is a former project in Belgium\footnote{\url{https://www.elia.be/en/infrastructure-and-projects/infrastructure-projects/auvelais-gembloux}} for which data was extracted from \textit{OpenStreetMap}. The resistances and an optimal path is shown in \autoref{fig:overview}. For both instances, realistic resistance values were decided together with Gilytics. Further details on the instances are given in \autoref{tab:instances}. 
\input{tables/instances}
\\ All experiments were executed on a compute node equipped with two 64-core AMD EPYC 7742 processors (2.25 GHz nominal, 3.4 GHz peak) and 512 GB of DDR4 memory, clocked at 3200 MHz. In consultation with the collaborator it was decided to assume a \textit{directed acyclic graph (DAG)} according to the criterion in \autoref{eq:edgecost}. Note that all proposed algorithms are also applicable for a general graph.

\subsection{Feasibility of large project regions}
While our general modelling approach is similar to~\cite{santos2019optimizing, piveteau2017novel}, our framework allows to process much larger instances due to efficient algorithms omitting the need for larger graph representations such as a line graph. Specifically, for a single path computation, the required memory mainly stems from the storage of distance and predecessor maps, together comprising $2m$ elements. For finding $k$ paths on the other hand, $4m$ elements are kept in RAM since two shortest path trees are computed. With respect to actual memory, distances and predecessors are represented as integer or floating points of $8$ bytes, resulting in $2m \cdot 8 = 16m$ bytes or $32m$ bytes respectively. In practice, $m$ depends on the project region and $d_{min}$ and $d_{max}$. Here, processing \textbf{instance~2} with $m \approx 857350 \cdot 632$ edges (see \autoref{tab:instances}) requires around $8.7$GB, but in our framework only edges reachable from the source are stored, such that the optimal path and even the $k$ shortest paths can be computed with a standard 8GB RAM. In contrast, if \textbf{instance~2} was processed with a line graph~\cite{santos2019optimizing, piveteau2017novel}, the number of edges and thus also the required memory would increase by a factor of 632. This demonstrates the large gain in efficiency achieved with our framework. On top of that, \textit{any} conceivable instance can be processed with our multi-scale approach. \autoref{fig:ksp_example} depicts $k$ shortest paths on \textbf{instance~1}, computed with the multi-scale approach, which would otherwise require $32 \cdot 7358107 \cdot 967 \approx 227$GB memory

\subsection{Comparison to an existing transmission line}
For \textbf{instance~1}, we extracted the pylon positions of an existing transmission line, denoted by $Q$ in the following. In \autoref{fig:ksp_example}, $Q$ is shown in contrast to five paths $P_1, \dots, P_5$ computed with our framework. Although we naturally do not know which geographic features were considered in the planning of $Q$, we can analyze $Q$ with respect to our data, as it is safe to assume that similar resistances were regarded when planning $Q$.  
\autoref{fig:costs} demonstrates that $Q$ is crossing many high resistance areas such as bird sanctuaries or even areas that are forbidden according to our data (buildings, nature reserves), while $P_1$ tends to follow roads and suitable regions around forest.
Even when \textit{omitting} costs for pylons that traverse forbidden regions, we find that the geographic resistances are reduced by 25\%, together with an 8\% reduction of angles, rendering $P_1$ a strictly better route. 
\begin{figure}[t]
\begin{subfigure}[t]{0.5\textwidth}
\vskip 0pt
    \includegraphics[width=\textwidth]{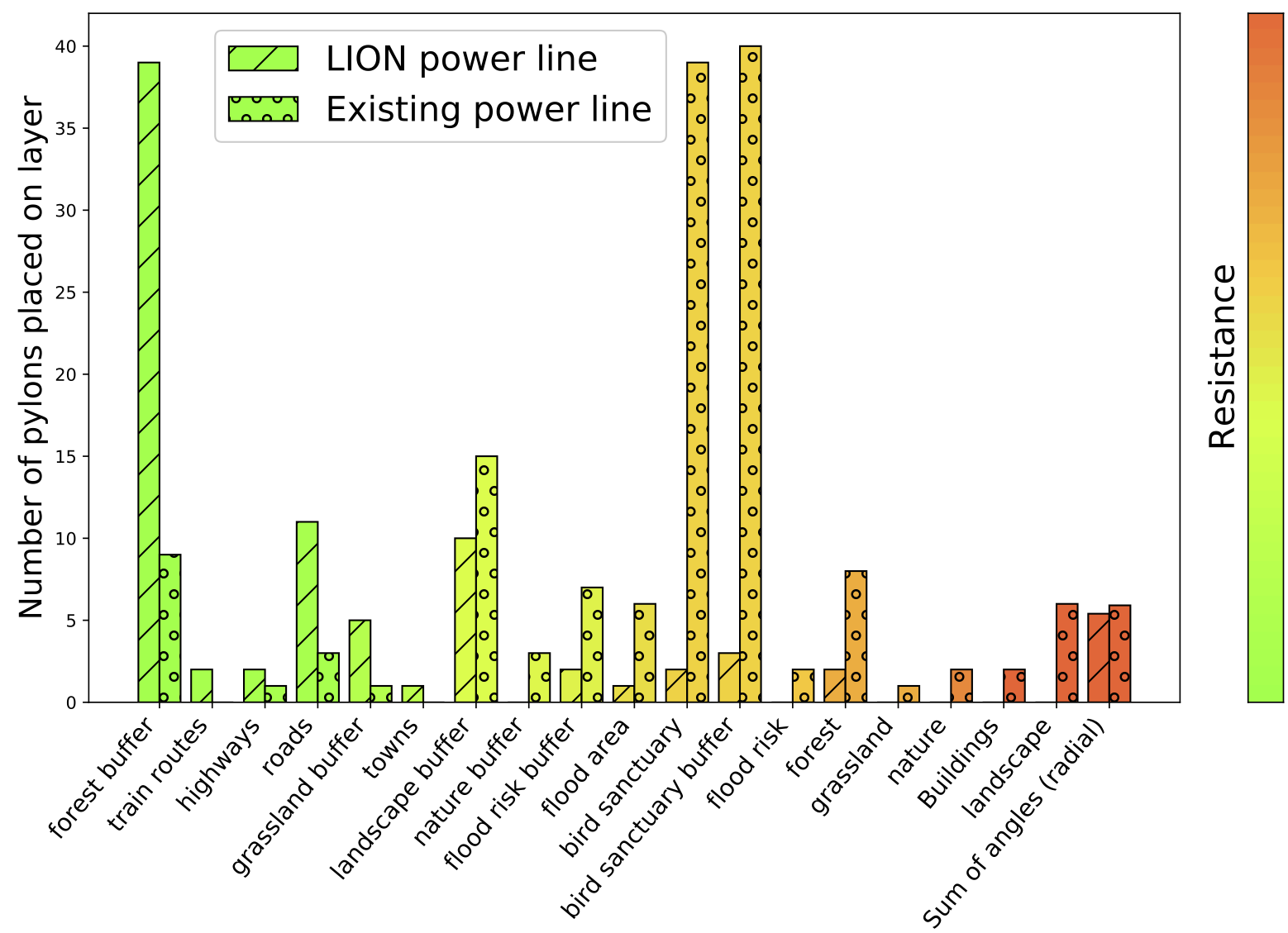}
    \caption{Comparison of resistance of the optimal path computed with LION and the existing power line. 
    Clearly, the existing power line is traversing areas of higher resistances than our path.}
    \label{fig:costs}
\end{subfigure}
\hfill
\begin{subfigure}[t]{0.48\textwidth}
\vskip 0pt
    \includegraphics[width=\textwidth]{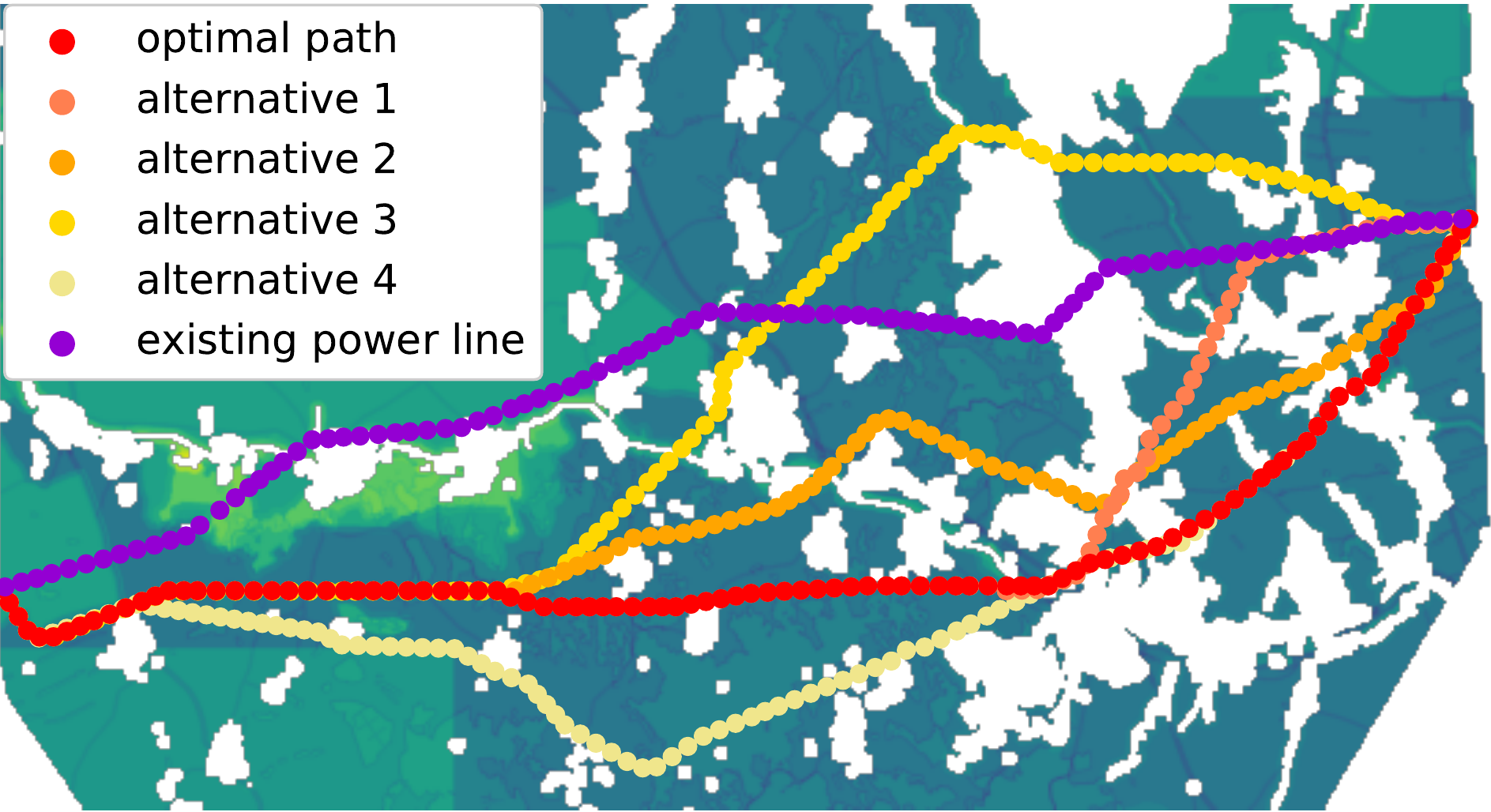}
    \vspace{2.8em}
    \caption{K diverse shortest paths for instance 1 (cropped) are shown, computed with the \textbf{find-ksp-max} method. Dark areas indicate low resistances, brighter areas higher resistance, and white areas are forbidden. High resistance areas are avoided by our paths in contrast to the existing power line}
    \label{fig:ksp_example}
\end{subfigure}
\caption{Qualitative and quantitative comparison of our results with existing infrastructure}
\end{figure}

\subsection{Pylon spotting versus line routing}
As explained above (\autoref{fig:graphbuild}), most related works~\cite{monteiro2005gis, vega1996image, west1997terrain, li2013optimization} compute a \textit{route} through the resistance raster and place transmission towers as a second step (called \textit{line-routing} in the following), instead of computing a globally optimal \textit{pylon spotting}. In \autoref{tab:baseline}, we compare the outputs of our pylon spotting to the line-routing approach. Specifically, we simply apply our algorithm on the graph with $d_{min}=1, d_{max}=1.5$, i.e. a grid-structured graph where each cell is connected to its 8-neighborhood. When the angle weight was greater than zero, we reduced also the angles of the line-route with our algorithm to provide a fair comparison of the angle costs. Based on this route, the pylons are spotted, where we even minimize the resistances again with our algorithm, but limiting the project region to locations along the route (\autoref{fig:prior}).
\input{tables/baseline_new}

Although our graphs are larger than grid-graphs by a few orders of magnitude, the instances can still be processed in less than 10 minutes\footnote{\textbf{instance~1} was down sampled to 20m resolution to allow for optimal path computation without the multi-scale approach}. The increase in runtime is, however, rewarded by the achieved reduction in costs: Resistances decrease by 10-20\% while the angles are reduced by more than 50\%. A major drawback of line-routing is also that minimizing the angle along the route does not necessarily reduce the angles at the pylons, but can in fact result in larger angles despite penalizing angles more (see \textbf{instance~2} angle weight 0.1). Furthermore, the routes are substantially different (see ''distance'' in \autoref{tab:baseline}), such that a local search around the route does not suffice to determine the optimal path.
\subsection{Efficiency of multi-scale processing}
We explored in different configurations how much the space and time could be reduced by our iterative multi-scale method (\autoref{sec:multiscale}), and at what cost. In \autoref{fig:pipeline_eval} the performances for \textbf{instance~2} are reported for the scale lists $[2,1], [3,1]$ and $[3,2,1]$, meaning for example that the instance is down-sampled to 30x30m resolution at first, then to 20x20m and last to 10m. The bars show the \textit{relative} increase or decrease of runtime, space and resistance, always compared to the direct computation ($[1]$).
\begin{wrapfigure}[19]{r}{10cm}
    \centering
    \vspace{1em}
    \includegraphics[width=0.55\textwidth]{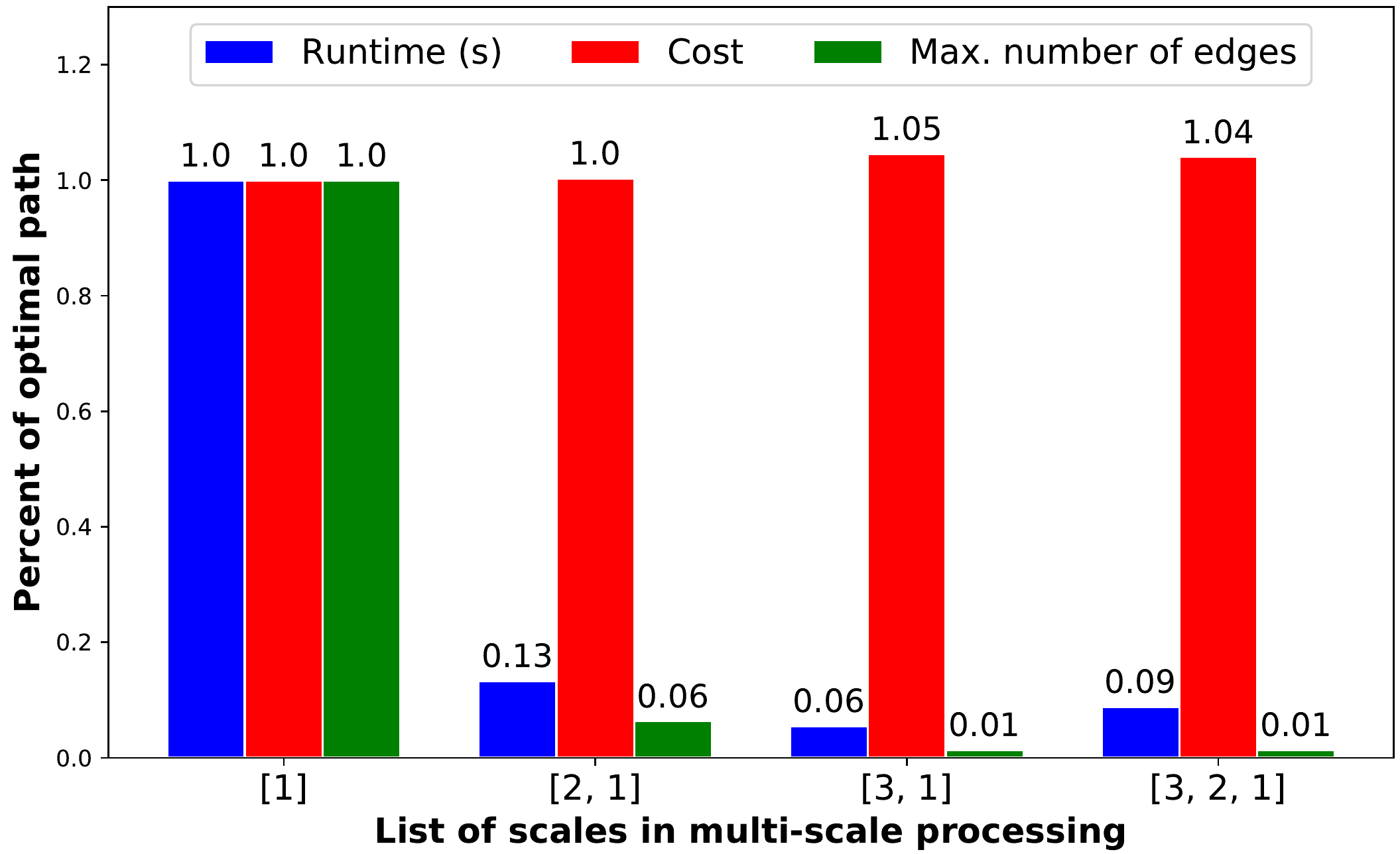}
    \label{fig:pipeline_eval_all}
    \caption{Evaluation of iterative multi-scale methods to reduce time and space requirements. The outputs were computed for three instances with four different list of scales ([1] refers to a direct optimization in one iteration). Space and runtime decrease significantly for all instances when using the iterative procedure, and the resistances increase only marginally as desired.}
    \label{fig:pipeline_eval}
\end{wrapfigure}
Importantly, the path resistances only increase by less than 5\%, while in some cases (e.g. $[2,1]$), the optimal path is found. Importantly, space time complexity is reduced dramatically due to the quartic dependence of the number of edges on the raster resolution. Specifically, down-scaling by factors $2$ and $3$ decreases the number of edges by factors $16$ and $81$, respectively. The overall runtime in all cases was below 13\% of the original runtime. 
Furthermore, we observed that having three scales, instead of two (e.g. $[3,2,1]$ vs. $[3,1]$) increases runtime, but also leads to slightly lower path resistances, indicating that more steps improve the ability to approximate the optimal path. These results, summarized in \autoref{fig:pipeline_eval}, demonstrate the potential of the multi-scale approach to find close-to-optimal paths at low computational costs.
%
\vspace{4em}
\subsection{Diversity-cost trade-off for k diverse shortest paths}
%
In section~\ref{ksp} several methods to compute the $k$ diverse shortest paths were discussed. 
Here we apply these methods on \textbf{instance~2} and vary the diversity threshold, observing the ability of each method to achieves high diversity at low path costs. 
Each point in \autoref{fig:kspcompare} corresponds to a set of 5 paths, and the diversity of these paths is plotted against the sum of their respective path costs. 
Optimally the points should be located in the lower right corner, such that they minimize costs but maximize diversity in the set of paths.\\ As expected, \textbf{find-ksp-max} and \textbf{corridor-penalizing} obtain best results with respect to the Yau-Hausdorff metric, while \textbf{greedy-set} and \textbf{k-dispersion} algorithms achieve high Jaccard distance at low costs. 
In short, guarantees are only given with \textbf{find-ksp-max}, but other algorithms such as \textbf{greedy-set} are shown to be effective and should be employed dependent on the diversity metric.

\begin{figure}[ht]
    \centering
    \includegraphics[width=0.8\textwidth]{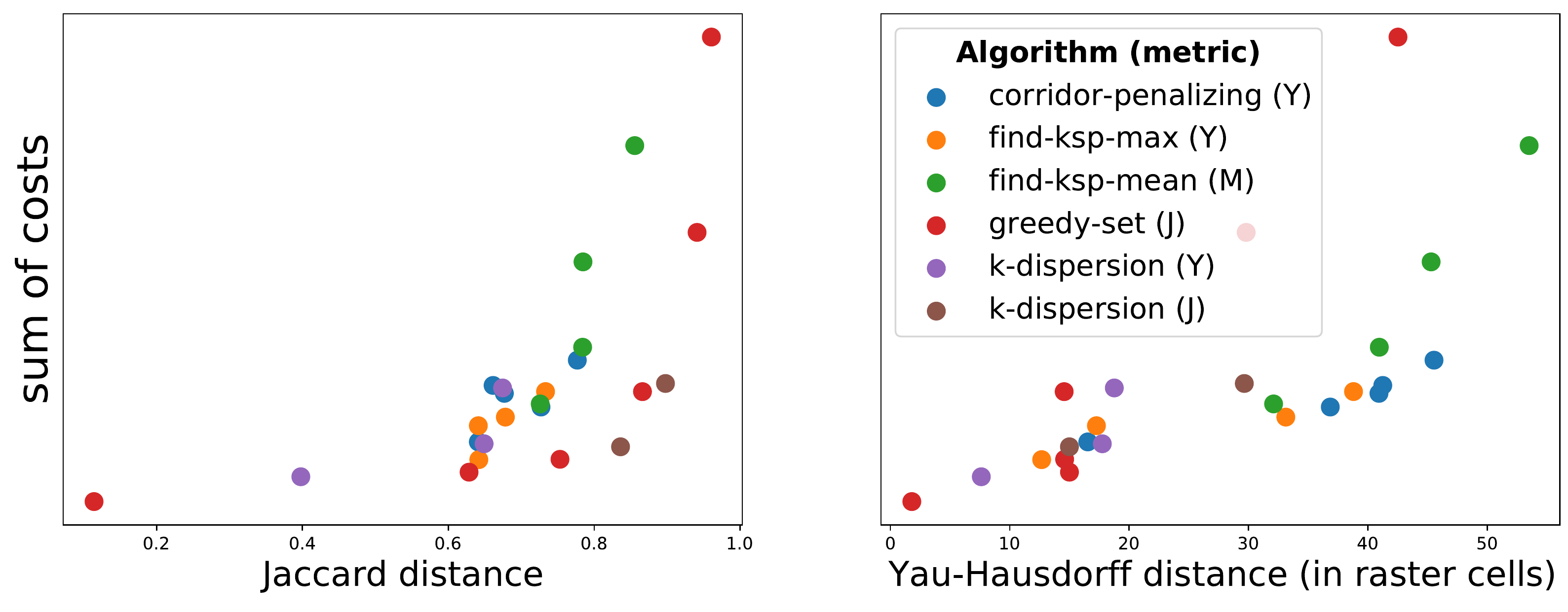}
    \caption{Diversity-cost trade-off of various algorithms (\textbf{instance~2}). Each point refers to a set of five paths that were computed with a certain algorithm and metric (Y - Yau-Hausdorff, J - Jaccard, M - Mean Euclidean distance). The \textbf{greedy-set} method minimizes the Jaccard distance, while \textbf{corridor-penalizing} or \textbf{find-ksp-max} perform best for Euclidean metrics.}
    \label{fig:kspcompare}
\end{figure}

%% file: tables/instances.tex
\begin{table}[h]
\centering
\resizebox{0.8\textwidth}{!}{%
\begin{tabular}{c|ccccccc}
\textbf{Instance} & \textbf{\begin{tabular}[c]{@{}c@{}}Raster\\ width\end{tabular}} & \textbf{\begin{tabular}[c]{@{}c@{}}Raster\\ height\end{tabular}} & \textbf{\begin{tabular}[c]{@{}c@{}}Project region\\  (not forbidden)\end{tabular}} & \textbf{d\_min} & \textbf{d\_max} & \textbf{\begin{tabular}[c]{@{}c@{}}neighbours\\ per node\end{tabular}} \\ \hline
\textbf{1}        & 4485                                                            & 3206          & \begin{tabular}[c]{@{}c@{}} 7 358 107 cells (51\%) \\ 735.8 km$^2$ \end{tabular}                                                                     & 25              & 35              & 967                                                                \\

\textbf{2}        & 1300                                                            & 739           & \begin{tabular}[c]{@{}c@{}}857 350 cells (89\%) \\ 85.7 km$^2$ \end{tabular}                                                                     & 15              & 25              & 632                                                                    
\end{tabular}%
}
\caption{Instance properties. All numbers refer to a raster of 10x10m cells.}
\label{tab:instances}
\end{table}

%% file: tables/baseline_new.tex
\begin{table}[ht]
\centering
\resizebox{0.8\textwidth}{!}{%
\begin{tabular}{ccccccccc}
\begin{tabular}[c]{@{}c@{}} \textbf{Instance} \\ (resolution) \end{tabular}
                   &
\multicolumn{1}{l}{\begin{tabular}[c]{@{}c@{}} \textbf{Angle} \\ \textbf{weight} \end{tabular}}
& 
\textbf{Approach}
& 
\multicolumn{1}{l}{\begin{tabular}[c]{@{}c@{}} \textbf{Time} \\ \textbf{(seconds)} \end{tabular}} 
& \multicolumn{1}{l}{\begin{tabular}[c]{@{}c@{}}\textbf{m (|E|)}\\ \textbf{(mio)}\end{tabular}} & \multicolumn{1}{l}{\begin{tabular}[c]{@{}c@{}} \textbf{Angle} \\ \textbf{cost} \end{tabular}} & \multicolumn{1}{l}{\textbf{Resistance}} & \multicolumn{1}{l}{
\begin{tabular}[c]{@{}c@{}}\textbf{Distance} \\ \textbf{(in m)}\end{tabular}} &                      \\ \hline 

\cellcolor[HTML]{F5F5F5}                                   & \cellcolor[HTML]{F5F5F5}                               & \cellcolor[HTML]{F5F5F5} \textbf{Line-routing} & 
\cellcolor[HTML]{F5F5F5} \textbf{34.1}  & 
\cellcolor[HTML]{F5F5F5} 3.6   &
\cellcolor[HTML]{F5F5F5} 27.0 &
\cellcolor[HTML]{F5F5F5} 24.3  &
\cellcolor[HTML]{F5F5F5}    \\


\cellcolor[HTML]{F5F5F5}                                   & \multirow{-2}{*}{\cellcolor[HTML]{F5F5F5}\textbf{0.0}} & \textbf{Ours}     & 197.4 & 220 & \textbf{22.8} & \textbf{20.5}  & 823  \\

\cellcolor[HTML]{F5F5F5}            &                               & 
\cellcolor[HTML]{F5F5F5}\textbf{Line-routing} & 
\cellcolor[HTML]{F5F5F5} \textbf{35.6}  & 
\cellcolor[HTML]{F5F5F5} 3.6   &
\cellcolor[HTML]{F5F5F5} 25.4 & 
\cellcolor[HTML]{F5F5F5} 25.4  & 
\cellcolor[HTML]{F5F5F5}    \\

\cellcolor[HTML]{F5F5F5}                                   & \multirow{-2}{*}{\textbf{0.1}} & \textbf{Ours}     & 196.8 & 220 & \textbf{7.7}  & \textbf{20.8}  & 943  \\

\cellcolor[HTML]{F5F5F5}                                   & \cellcolor[HTML]{F5F5F5}                               & \cellcolor[HTML]{F5F5F5}\textbf{Line-routing} & 
\cellcolor[HTML]{F5F5F5} \textbf{33.1}  &
\cellcolor[HTML]{F5F5F5} 3.6   &
\cellcolor[HTML]{F5F5F5} 17.9 & 
\cellcolor[HTML]{F5F5F5} 28.2  & 
\cellcolor[HTML]{F5F5F5}    \\
\multirow{-6}{*}{\cellcolor[HTML]{F5F5F5}\begin{tabular}[c]{@{}c@{}}\textbf{Inst. 1}\\ (20 m)\end{tabular}} & \multirow{-2}{*}{\cellcolor[HTML]{F5F5F5}\textbf{0.3}} & \textbf{Ours}     & 208.2 & 220 & \textbf{2.6} & \textbf{21.6 } & 1414 \\
\hline
&                                & \cellcolor[HTML]{F5F5F5} \textbf{Line-routing} &
\cellcolor[HTML]{F5F5F5}\textbf{23.8}  &
\cellcolor[HTML]{F5F5F5}2.2   & 
\cellcolor[HTML]{F5F5F5} \textbf{8.0} & \cellcolor[HTML]{F5F5F5}
109.9 &\cellcolor[HTML]{F5F5F5}    \\

& \multirow{-2}{*}{\textbf{0.0}} & \textbf{Ours}     & 594.0 & 286 & 11.0 & \textbf{98.6} & 263  \\

   & \cellcolor[HTML]{F5F5F5}                               & \cellcolor[HTML]{F5F5F5}\textbf{Line-routing} & 
   \cellcolor[HTML]{F5F5F5}\textbf{21.9}  & 
   \cellcolor[HTML]{F5F5F5}2.2   & 
   \cellcolor[HTML]{F5F5F5}9.6  & 
   \cellcolor[HTML]{F5F5F5}118.4 & 
   \cellcolor[HTML]{F5F5F5}   \\

  & \multirow{-2}{*}{\cellcolor[HTML]{F5F5F5}\textbf{0.1}} & \textbf{Ours}     & 588.5 & 286 &\textbf{2.7}  & \textbf{101.5 } & 482  \\

 &                               & \cellcolor[HTML]{F5F5F5}\textbf{Line-routing} &
 \cellcolor[HTML]{F5F5F5}\textbf{20.0}  & 
 \cellcolor[HTML]{F5F5F5}2.2   & 
 \cellcolor[HTML]{F5F5F5}6.8  & 
 \cellcolor[HTML]{F5F5F5}122.7 & 
 \cellcolor[HTML]{F5F5F5}    \\
 
\multirow{-6}{*}{\begin{tabular}[c]{@{}c@{}}\textbf{Inst. 2}\\ (10 m) \end{tabular}} & \multirow{-2}{*}{\textbf{0.3}} & \textbf{Ours}     & 583.4 & 286 & \textbf{1.2}  & \textbf{105.4} & 498

\end{tabular}
}
\caption{Comparison of pylon-spotting to line routing (best value marked bold). ''Distance'' is the average distance between the pylons of our approach compared to the baseline. Although our approach leads to much larger graphs, the optimal path can still be computed in less than 10 minutes and in all cases leads to significantly lower angle and resistance costs.}
\label{tab:baseline}
\end{table}

%% file: chapter/06_discussion.tex
\vspace{3em}

\section{Discussion}
 
The layout of power transmission lines is a contentious topic with many stakeholders involved, calling for software and mathematical methods to support the planning process. 
In contrast to previous work, we presented an efficient framework for the \textit{global optimization} of transmission line layout, for computing (angle-) optimal pylon spotting. On the one hand, we view our methods as a mathematically sound way to propose a route that is optimal with respect to the input parameters, while on the other hand it can also be seen as a decision support system for planners, with its ability to output diverse shortest paths for comparison.   

While our experiments show the system's ability to reduce resistances and angle costs compared to previous work, further work should aim at accelerating the proposed algorithms with approximation algorithms or parallelization. Many further challenges remain, including for example the optimization of pylon heights in more complex (non-flat) terrain.
\begin{wrapfigure}[15]{r}{9.5cm}
    \centering
    \includegraphics[width=0.55\textwidth]{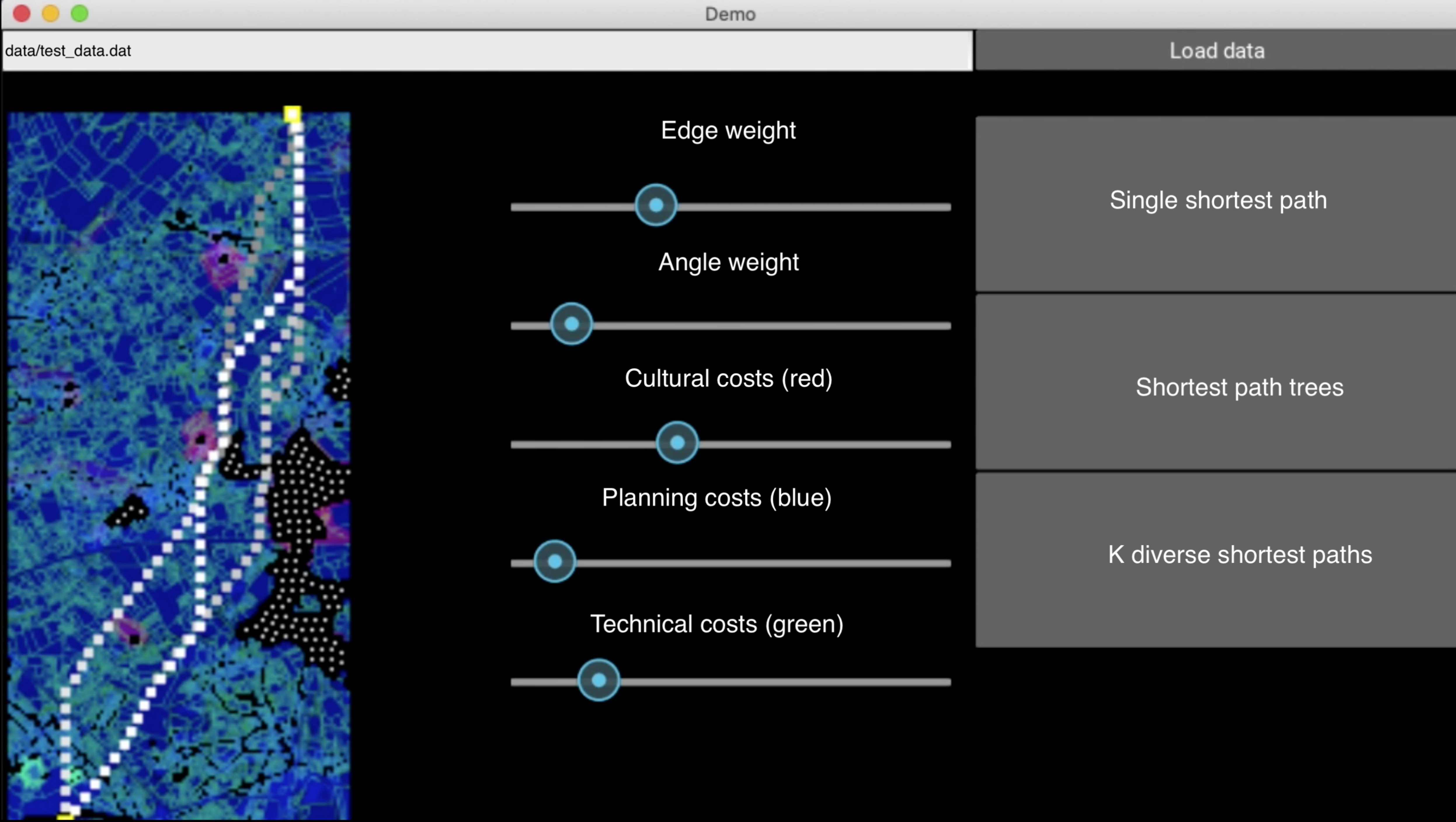}
    \caption{User interface for power infrastructure planning}
    \label{fig:ui}
\end{wrapfigure}
Alongside a \texttt{PyPi} package \textit{lion-sp}, we also release a small graphical user interface\footnote{Installation instructions available on GitHub \url{https://github.com/NinaWie/PowerPlanner}}, shown in \autoref{fig:ui}, 
demonstrating how a decision support tool can be created by adding appropriate UI/UX components.
Moreover, the proposed methodology be be extended beyond transmission line planning. Other applications such as the layout of pipes, underground transmission lines, or highways could benefit from the same optimization framework.
We thus hope to pave the road towards software-assisted linear infrastructure planning, backed by proven mathematical methods, 
for easier collaboration between local residents, planning agencies, environmentalists and construction authorities. 

%% file: chapter/appendix.tex
\section{Accelerating minimal-angle Bellman-Ford for step-functions as angle cost functions}\label{sec:discrete}

In the proposed minimal-angle Bellman-Ford algorithm (\autoref{alg:alg1}), the edge distances $D[e]$ were updated by taking the minimum edge- plus angle-cost over the incoming edges. 
At a vertex with $k$ incoming edges and $l$ outgoing edges, the algorithm requires $k\cdot l$ operations to update the distances at all its outgoing edges, since the minimum is taken for each outgoing edge. Note that these operations appear at each vertex regardless of the type of graph (DAG, directed or undirected). Here, we propose a procedure that decreases the number of update operations significantly if the angle cost function is a step function:
\begin{theorem}\label{discrete}
If $c_a: [0,180]\longrightarrow S$ is a step function with $|S|$ possible angle costs ($|S|<\delta_v^-, \delta_v^+$), the shortest walk can be computed in time $\mathcal{O}\big(p \cdot \sum_v (|S|\cdot \delta_v^- + \delta_v^+) \cdot \log (|S| \delta_v^- \delta_v^+)\big)$.
\end{theorem}
As proof, we will explain an algorithm and show its correctness and runtime. To simplify notation, we will from now on only consider one vertex with $k$ incoming edges and $l$ outgoing edges. We show that in the case of a step-function $\mathcal{O}(|S| \cdot k + l \cdot \log (|S|kl))$ operations suffice to update one vertex (previously $k\cdot l$), which leads to the improved runtime of minimal-angle Bellman-Ford as stated in Theorem~\ref{discrete}.

We denote the incoming edge vectors by the set $E^- = \{e_1^-, \dots, e_k^-\}$ and the outgoing edge vectors by $E^+ = \{e_1^+, \dots, e_l^+\}$. Note that if the graph is undirected, it is simply transformed into a directed graph, such that each original edge appears once in $E^-$ and once in $E^+$, and $k=l$. 
Thus, from now on a directed graph is assumed. Furthermore, each incoming/outgoing edge reaches/leaves the vertex $v$ in a certain angle. 
This angle is computed for all edges with respect to an arbitrary direction $\vec{w}$, yielding values $\alpha_i = \sphericalangle(\vec{w}, e^-_i)$ for $e^-_i\in E^-$ and $\beta_j = \sphericalangle(\vec{w}, e^+_j)$ for $e^+_j\in E^+$. 
Crucially, we define $\sphericalangle(x)$ as the angle in clockwise direction from 0 to 360 degrees, instead of a direction-insensitive angle from 0 to 180 degrees. Only with angle values from 0 to 360 it holds that $\alpha_i = \beta_j$ if and only if the vectors $e^-_i$ and $e^+_j$ form a straight line ($e^+_j$ leaves $v$ in the same direction as $e^-_i$ reaches it).

We further define the angle cost function $c_a$ in the following manner: 
\begin{gather*}
c_a(x) = 
    \begin{dcases}
        \theta_1 & 0 \leq x < \gamma_1\\
        \theta_2 & \gamma_1 \leq x < \gamma_2 \\
        \theta_3 & \gamma_2 \leq x < \gamma_3\\
        \cdots \\
        \theta_{|S|} & \gamma_{|S|-1} \leq x \leq 180 \\
    \end{dcases}
\end{gather*}
In other words, between an angle of $\gamma_{s-1}$ and $\gamma_{s}$, the angle cost is $\theta_s$. Thus, $\theta_1, \dots, \theta_{|S|}$ is the co-domain of the angle cost function $c_a$.

In the following algorithm we exploit that if there are only $|S|$ different angle costs, then only $k \cdot |S|$ different values are possible for the distance of an arbitrary outgoing edge from the source. The set of possible path distances is finite and can be enumerated as $L = \{ D[e^-_1] + \theta_1, D[e^-_1] + \theta_2, \dots, D[e^-_1] + \theta_{|S|}, \dots, \dots, D[e^-_k] + \theta_1, D[e^-_k] + \theta_2, \dots, D[e^-_k] + \theta_{|S|}\}$. The goal is to determine the optimal predecessor for each outgoing edge in $E^+$, by selecting the lowest of these values which is feasible in the respective angle.

Our algorithm proceeds in the following manner: First, the angles of the outgoing edges $\beta_1, \dots, \beta_l$ are inserted in a balanced and sorted AVL tree ($\mathcal{O}(l \log l)$). Next, the possible path distance values in $L$ are sorted. Sorting $|L| = k|S|$ values requires $\mathcal{O}(k |S| \cdot \log (k\cdot |S|))$. More specifically, not the distances themselves are sorted, but the corresponding tuples $(i,s)$ of the incoming edge index together with the index of the angle cost. 
The algorithm then iterates over the sorted tuples, such that $(i^*, s^*) = \argmin_{i,s} D[e^-_i] + \theta_s$ is retrieved first. 
In each iteration, the feasible range of angles given by the bounds $\alpha_i + \gamma_{s-1}$ to $\alpha_i + \gamma_{s}$ and $\alpha_i - \gamma_{s}$ to $\alpha_i - \gamma_{s-1}$ is considered. All outgoing edges whose angle value $\beta$ lies in this range will have the incoming edge $i$ as its predecessor, if they have not been assigned a predecessor previously.

In order to retrieve the relevant outgoing angles in this range, we first define an auxiliary function \texttt{get-inbetween} (\autoref{alg:getinbetween}) that retrieves all values from a tree which lie between a lower and upper bound. Note that this can be efficiently implemented with a recursive function.

\begin{algorithm}[H]
\SetAlgoLined
\Indp{Input: AVL tree $T$, current list of values $V$, lower bound $x_l$, upper bound $x_u$} \\
\uIf{T.isEmpty()}{
    return \{\}
}
$n = T.root()$ \\
\tcp{If value was inbetween or too large, go left}
\uIf{$x_l \leq n.key()$}{
    $T_l \leftarrow$ sub-tree rooted in left child of $n$\\
    V = $V \cup $ get-inbetween($T_l, V, x_l, x_r$)
}
\tcp{If value was inbetween or too low, go right}
\uIf{$n.key() \leq x_r$}{
    $T_r \leftarrow$  sub-tree rooted in right child of $n$\\
    V = $V \cup $ get-inbetween($T_r, V, x_l, x_r$)
}
\tcp{Add tree node if the value is inbetween $x_l$ and $x_r$}
\uIf{$x_l \leq n.key() \leq x_r$}{
    $V = V \cup n$ \\
}
\Return $V$
\caption{\texttt{get-inbetween}}
\label{alg:getinbetween}
\end{algorithm}

The main algorithm that was described above is given in \autoref{alg:discrete}.

\begin{algorithm}[H]
\SetAlgoLined
\Indp{Input: $E^+$, $E^-$, $\alpha_i\ \forall i, \beta_j \forall j$} \\
\tcp{Insert outgoing angle values in AVL tree:}
$T \leftarrow$ empty AVL tree\;
\For{$j=1..l$}{
    T.insert($\beta_j$)
}
\tcp{Sort incoming edge distances plus angle costs}
$L = [(1,1), \dots, (1,|S|), \dots, \dots, (k,1), \dots, (k,|S|)]$\;
$L^* = $ sort tuples $(i,s)$ in $L$ by their distance value $D[e^-_i] + \theta_s$ \;
\tcp{Iterate over list starting with the minimum}
\For{$(i,s)\in L^*$}{
    $A_1 \leftarrow $ \texttt{get-inbetween}($T, \{\}, \alpha_i + \gamma_{s-1}, \alpha_i + \gamma_{s}$)\;
    \tcp{Add nodes in the other direction}
    $A \leftarrow $ \texttt{get-inbetween}($T, A_1, \alpha_i - \gamma_{s}, \alpha_i - \gamma_{s-1}$)\;
    \tcp{For all retrieved nodes}
    \For{$\beta_j\in A$}{
        T.delete($\beta_j$)\;
        \tcp{update optimal predecessor}
        $P[e^+_j] = e^-_i$\;
    }
}
\caption{\texttt{Step-function update algorithm}}
\label{alg:discrete}
\end{algorithm}

Note that special care has to be taken with respect to the circularity of the $x$-axis here: As the $x$-values are angles, the highest value is actually close to the lowest value in the tree. Thus, if $a_i+\gamma_s > 360$, then \texttt{get-inbetween} must be called two times, once with the range of $a_i+\gamma_{s-1}$ to 360, and once with 0 to $(a_i+\gamma_s) \text{ mod } 360$, where mod denotes the modulo operation. For the sake of simplicity, in \autoref{alg:discrete} such special cases are not covered, and will not be covered in the following proofs. Similarly, we do not regard the case of several outgoing edges in the same angle (and thus the same $\beta$ value). Extensions of the above algorithms that regard the border cases and non-unique values are straightforward to implement.

\begin{lemma}
Algorithm~\ref{alg:discrete} assigns the correct predecessor to each outgoing edge.
\end{lemma}
\begin{proof}
Let $e^+_j$ denote any outgoing edge. If the corresponding angle value $\beta_j$ is retrieved from $T$, this means $\beta_j$ lies between two values $\alpha_i +\gamma_{s-1}$ and $\alpha_i +\gamma_{s}$ due to the correctness of \texttt{get-inbetween} (or between $\alpha_i - \gamma_{s}$ and $\alpha_i - \gamma_{s-1}$, but the proof is analogous). Due to the computation of $\alpha$ and $\beta$, the fact that $\alpha_i +\gamma_{s-1} \leq \beta_j \leq \alpha_i +\gamma_{s}$ implies that $\gamma_{s-1} \leq \sphericalangle(e^-_i, e^+_j) \leq \gamma_s$ and the distance thus amounts to $D[e_i^-] + c_a(\sphericalangle(e^-_i, e^+_j)) = D[e_i^-] + \theta_s$.\\
Furthermore, observe that each outgoing edge is retrieved from the tree only once, since its entry in $T$ is deleted afterwards. When $e^+_j$ is retrieved from $T$ in an iteration where the current tuple is ($i,s$), then $e^-_i$ is assigned as its predecessor and the predecessor is not changed afterwards. For the sake of contradiction, assume there exists a better predecessor for $e^+_j$, say $e^-_{\hat{i}}$. Then this predecessor would imply a lower distance, so $D[e^-_{\hat{i}}] + \theta_{\hat{s}} < D[e_i^-] + \theta_s$ with some other angle value $\hat{s}$. The tuple $(\hat{i}, \hat{s})$ must therefore appear earlier in $S^*$, because the tuples in $S^*$ are sorted by their corresponding cost. However, if such a better tuple existed, $\beta_j$ would be between the corresponding bounds given by $\gamma_{\hat{s}-1}$ and  $\gamma_{\hat{s}}$, and it would have been deleted from the tree earlier. The contradiction thus shows that the first tuple that retrieves $\beta_j$ from the tree indeed declares the optimal predecessor for $e^+_j$.
\end{proof}

\begin{lemma}
Algorithm~\ref{alg:discrete} requires  $\mathcal{O}\big((|S|k + l) \cdot \log (|S| k l)\big)$ operations to update the predecessors of all outgoing edges.
\end{lemma}
\begin{proof}
The computation of the angle values $\alpha_1, \dots \alpha_k$ and $\beta_1, \dots, \beta_l$ is straightforward and requires only one pass over the incoming and outgoing edges respectively. Inserting all $\beta_j (j=1..l)$ in an AVL tree is feasible in $\mathcal{O}(l \log l)$. Then each incoming edge is paired with all $S$ possible angle costs and these tuples are sorted, which takes $\mathcal{O}(k|S| \log k|S|)$ steps. In the subsequent loop of $k|S|$ iterations, the nodes are retrieved from the tree, and importantly, once a node was retrieved it is always deleted. In a single iteration, there could be $\mathcal{O}(l \log l)$ operations, for retrieving and deleting all elements from $T$. However, the total number of steps, summed over all iterations, is upper bounded by $\mathcal{O}(l \log l)$ as well because every element can only be encountered once. Since \texttt{get-inbetween} is executed in each iterations (using at least $\mathcal{O}(\log l)$ steps to find the closest values to $x_l$ and $x_r$), there are $\mathcal{O}(k|S| \log l +  l \log l)$ steps in total in the main loop.\\
Together, $\mathcal{O}(l \log l) +\mathcal{O}(k|S| \log k|S|)  + \mathcal{O}(k|S| \log l +  l \log l)$ can be summarized as $\mathcal{O}((k|S| + l) \log |S|kl)$.
\end{proof}

This result concludes the proof of Theorem~\ref{discrete}, when replacing $k$ by $\delta^-(v)$, $l$ by $\delta^+(v)$. Specifically, we use \autoref{alg:alg1} to compute the shortest path, but employ \autoref{alg:discrete} introduced here as an improved update rule when the angle cost function is a step function. Note that the runtime is only improved if $|S|<<\delta^+(v)$ since otherwise $|S|\delta^-(v) \approx \delta^+(v)\delta^-(v)$ and thus the simple update rule with complexity $\delta^+(v)\delta^-(v)$ would suffice.

\section{An accelerated update algorithm for convex angle cost functions}\label{convex}

As in the discrete case above, we show that the runtime decreases 
for any \textbf{convex} increasing angle cost function. The algorithm will be presented in the following.

\begin{theorem}
Let $c_a: [0,180] \longrightarrow \mathbb{R}^+$ be a convex and monotonically increasing angle cost function. Then the runtime of \autoref{alg:alg1} reduces to $\mathcal{O}\Big(p\cdot\sum_v  \big(\delta^-(v)+\delta^+(v)\big) \log \big(\delta^+(v) \delta^-(v)\big) \Big)$, since at each vertex with $k$ incoming and $l$ outgoing edges only $\mathcal{O}\big((k+l) \log kl\big)$ operations are executed.\\
\end{theorem}

We define $E^-$, $E^+$ and the corresponding in and outgoing angle values $\alpha_1, \dots, \alpha_k$ and $\beta_1, \dots, \beta_l$ as above in appendix~\ref{sec:discrete}.
With this notation, we can formulate the problem to determine the the optimal predecessor in terms of angle and distances for each outgoing edge as $$\forall j:\ \text{ Find } \argmin_i D[e_i^-] + c_a(|\alpha_i - \beta_j|_{m})\ ,\ \ |x|_{m} = \min \{|x|, 360-|x|\}\ .$$
The operation $|x|_{m}$ computes the smaller angle between both and thereby ensures that the input of $c_a$ lies within the function domain $[0,180]$. 
\begin{figure}[ht]
    \centering
    \includegraphics[width=\textwidth]{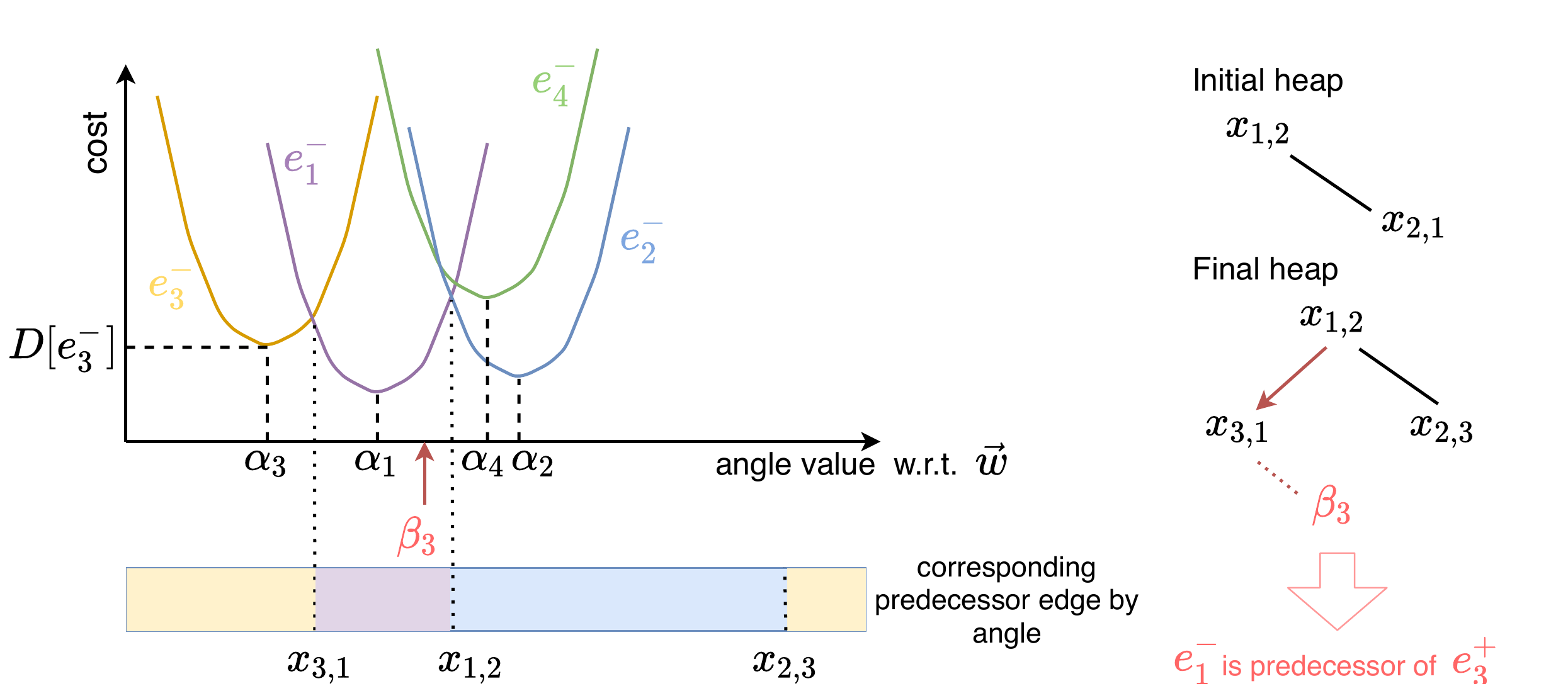}
    \caption{Angle update algorithm for convex, monotonically increasing angle cost functions. The intuition is that the intersections of neighboring functions define the borders of the ''area of responsibility''for each incoming edge. Intersections are iteratively added to a tree, starting with the incoming edge with lowest base cost $D$. The tree is sorted and balanced. After the tree is fully constructed, each outgoing edge $e^+_j$ is assigned its optimal predecessor by finding the closest value to $\beta_j$ in the tree.}
    \label{fig:convex}
\end{figure}

Now $D[e_i^-] + c_a(|\alpha_i - \beta_j|_m)$ can be seen as a function associated to the incoming edge $i$ and taking an angle $\beta_j$ as input. We denote these functions by $f_{i}$ such that $f_{i}(x) =  D[e_i^-] + c_{a}(|\alpha_i-x|_m)$. Observe that each function is a shifted version of $c_a(|x|_m)$, displaced by $\alpha_i$ in x- and $D[e_i^-]$ in y-direction. \autoref{fig:convex} shows an example, where each incoming edge induces a cost function for the angles of possible outgoing functions. Note that $f_{i}$ is symmetric around $\alpha_i$, and it has a global minimum at $\alpha_i$ because $c_a$ is monotonically increasing. Also, if the x-axis is rotated appropriately such that $|x|_m = |x|$ then $f_i$ is convex because it is a (shifted) composition of two convex functions, $c_a$ and $|x|$. The key observation used in the following algorithm is that - due to convexity - an incoming edge is the optimal predecessor for a \textit{closed} range of angles (without gaps). In other words, if two outgoing edges at angles $\beta_1, \beta_2$ have the same optimal predecessor, then another outgoing edge with angle $\beta_3$ that fulfills $\beta_1\leq \beta_3 \leq \beta_2$ must have the same predecessor. To determine the ''area of responsibility'', i.e. the range of angles in which an outgoing edge would get $e^-_i$ assigned as its predecessor (\autoref{fig:convex} bottom left), the intersections of the functions $f_{i}$ are computed. The idea is to construct a tree $T_x$ containing the intersection values (\autoref{fig:convex} right) that can be traversed in $\mathcal{O}(\log k)$ in the end for each outgoing edge to find the closest intersection and thereby its optimal predecessor edge.

First, the incoming edges are sorted by their distance from the source $D[e^-_i]$. For simplicity of notation, we from now on assume that they are already sorted, such that $D[e^-_1]\leq D[e^-_2]\leq \dots \leq D[e^-_k]$. Then, the incoming edges are successively processed and the intersection of their respective function $f_i$ with previously responsible functions in their range is computed and added to a balanced AVL tree $T_x$.
\autoref{alg:convex} describes the procedure in detail. The tree is built iteratively by insertion and replacement of intersection points. We define an operation \texttt{compute-intersection} that outputs the intersection of two functions. Given $f_{p}, f_{q}$, \texttt{compute-intersection}($f_{p}, f_{q}$) yields a $x$ value at which the functions intersect. 
Crucially, the angle values are circular, i.e. an angle of 0 degrees corresponds to an angle of 360 degrees. Therefore, \texttt{compute-intersection}($f_{p}, f_{q}$) $\neq$\texttt{compute-intersection}($f_{q}, f_{p}$) because if the two functions intersect at any point $x$, then they intersect at least at two points, $x$ and $360-x$. Thus, for each intersection we store which function was lower in clockwise (positive x-) direction and which one in counter-clockwise (negative x-) direction. For example, a triple $(x,p, q)$ means that 
$e^-_{p}$ is the optimal predecessor for $\beta$-angles that are close to $x$ in counter-clockwise direction and $e^-_{q}$ is optimal for outgoing edges close to $x$ in clockwise direction. Due to the circularity it is for example possible that $\alpha_p = 300$, $\alpha_{q} = 10$ and $x=330$. It will be argued further below that \texttt{compute-intersection} can be approximated efficiently. 

Secondly, \autoref{alg:convex} utilizes an operation called \texttt{find-closest} that returns the two closest value in a balanced AVL-tree. The algorithm is outlined in detail in \autoref{alg:findclosest}. The balanced and sorted tree is traversed and the closest left and right values $\sigma_l$, $\sigma_r$ with respect to the input element are memorized and returned.\\
Last, in addition to the sorted tree of intersections $T_x$ we also maintain a double-linked list $L$ of neighboring intersections. If $(x, p, q)$ is the intersection triple of $f_p$ and $f_q$, then $(x,p,q)$.next is defined as the next intersection in clockwise direction which must therefore be of the form $(y, q,r)$ with some $r$. Similarly, $(x,p,q)$.previous is of the form $(x,o,p)$ and denotes the next intersection counter-clockwise.
%
In the main algorithm detailed in \autoref{alg:convex}, for each new incoming edge $e^-_i$ the closest angle values of incoming edges that were already processed are found with \texttt{find-closest}. To do this efficiently in logarithmic time, all $\alpha_i$ are iteratively inserted into a second tree $T_{\alpha}$ and \texttt{find-closest}($T_{\alpha}, \alpha_i$) is called.

\begin{algorithm}[H]
\SetAlgoLined
\Indp{Input: AVL tree $T$ element $x$} \\
\tcp{define $\sigma_l$ and $\sigma_r$ as the currently closest elements}
$\sigma_l, \sigma_r = \infty$ \\
$n = T.root()$ \\
\tcp{Traverse tree}
\While{not isLeaf$(n)$}{
    \uIf{$x > n.key()$}{
        $\sigma_l = n$ \\
        $n = \text{ getRightChild}(n)$
    }\uElse{
        $\sigma_r = n$ \\
        $n = \text{ getLeftChild}(n)$ \\
    }
}
\tcp{check border cases if one is still infinity}
\uIf{$\sigma_l\text{ is }\infty$}{
    \tcp{set left bound to the largest element in the tree}
    $\sigma_l = T[-1]$
}
\uIf{$\sigma_r\text{ is }\infty$}{
    \tcp{set right bound to the smallest element in the tree}
    $\sigma_r = T[0]$
}
\Return $\sigma_l, \sigma_r$
\caption{\texttt{find-closest}}
\label{alg:findclosest}
\end{algorithm}


In short, algorithm~\ref{alg:convex} implements the following steps:
At the $i$-th step, $f_i$ is retrieved, thus corresponding to the $i$-th lowest incoming cost $D[e^-_i]$. Then \texttt{find-closest} is called for $a_i$. Let the result, i.e. the closest angles to $\alpha_i$, be denoted by $\alpha_p$ and $\alpha_q$. If $f_i$ obtains a better cost at the intersections between $f_p$ and $f_q$, the triple is deleted from $T_x$ and the next triples in both directions are checked. This is repeated until encountering an intersection point where $f_i$ can not improve on the previous functions. Due to convexity, once $f_i$ is outperformed by another function it will not be lower again in this direction. \\
Once these functions that mark the end of the area of responsibility of edge $i$ have been determined, the intersections with them are computed and added to $T_x$. Also, the double-linked list is updated accordingly. 
\\After all incoming edges were considered and possibly added to $T_x$, the outgoing edges can be processed. For each $\beta_j$, the left and right closest intersection triples $(x_1, -, i)$ and $(x_2, i, -)$ are found with a call to \texttt{find-closest} and the corresponding incoming edge $e^-_i$ is assigned as the predecessor of $e^+_j$.

Consider the example in \autoref{fig:convex}. First of all, the first function must be found that intersects with $f_1$ (PART 1 in \autoref{alg:convex}). Clearly, $f_1(\alpha_1)\leq f_i(\alpha_i)\ \forall i$, so it must be checked whether $f_2$ is better at any other point, formally whether $\exists x: f_1(x)>f_2(x)$. Lemma~\ref{firstloop} will show that it is sufficient to check whether $f_2$ is lower at the opposite point of $\alpha_1$, namely $(\alpha_1 + 180) \text{ mod } 360$, where $f_1$ is maximal. Since $f_2$ satisfies this condition, the two intersection (clockwise and counter-clockwise from $\alpha_1$) between $f_1$ and $f_2$ are computed and 
added to the tree as shown in \autoref{fig:convex} (upper right).
$T_x$ is sorted and balanced. In the first iteration of PART 2 in \autoref{alg:convex}, $\alpha_3$ is considered and its neighboring $\alpha_i$ are determined with \texttt{find-closest}. In the example, $\alpha_3$ is between $\alpha_2$ counter-clockwise and $\alpha_1$ clockwise. 
We therefore check whether $f_3$ improves the value at intersection $x_{2,1}$. We indeed find that $f_3(x_{2,1}) < f_1(x_{2,1})$, so there is clearly a range of angles where $e^-_3$ is the optimal predecessor. The triple $(x_{2,1}, 2, 1)$ is thus deleted from $T_x$. To determine the range where $f_3$ is optimal, the \autoref{alg:convex} further proceeds by checking the next intersections in clockwise and counter-clockwise directions. In this case, the next intersection in both cases is $x_{1,2}$. However, $f_3(x_{1,2}) > f_1(x_{1,2})$, so the algorithm can stop here. After these checks, in \autoref{alg:convex} now $p^*=2$ for our example, and $q^*=1$. We thus compute the new intersections $x_{2,3}$ and $x_{3,1}$ and add the triples $(x_{2,3}, 2, 3)$ and $(x_{3,1}, 3, 1)$ to $T_x$. 
On the other hand, in the last iteration $f_4$ (whose angle $\alpha_4$ lies between $\alpha_1$ and $\alpha_2$) does not outperform the value at the intersection between $f_1$ and $f_2$  
and consequently $\alpha_4$ is not an optimal predecessor for any outgoing angle. After all incoming edges were considered and $T_x$ is final, the $\beta$-values of the outgoing edges are processed. In \autoref{fig:convex} for example $\beta_3$ is between the intersection triples $(x_{3,1}, 3, 1)$ and $(x_{1,2}, 1,2)$ and thus $e^-_1$ is the optimal predecessor for $e^+_3$.

\footnotesize{
\begin{algorithm}[H]
\SetAlgoLined
\tcp{PART 1: Find first intersecting functions}
 $i_{start}=2$\\
 $\hat{\alpha_1} \longleftarrow (\alpha_1 + 180) \text{ mod } 360$ 
 \tcp{opposite point to $\alpha_1$} 
 \While{$i_{start}\leq k$ and $f_{i_{start}}(\hat{\alpha_1}) > f_{1}(\hat{\alpha_1})$}{
    $i_{start} \longleftarrow i_{start} + 1$ \\
 }
 \uIf{$i_{start}>k$}{
    \tcp{Break: no function better than $f_1$ at any point}
    $T_x$.insert($\alpha_1, 1,1$) \\
    \Return $T_x$ 
 }
 $x_1$ = \texttt{compute-intersection}($f_1, f_{i_{start}}$) \\
 $x_2$ = \texttt{compute-intersection}($f_{i_{start}}, f_1$) \\
 $(x_1, 1,i_{start})$.next = $(x_1, 1,i_{start})$.previous = $(x_2, i_{start},1)$ \\
 $(x_2, i_{start}, 1)$.next =$(x_2, i_{start}, 1)$.previous = $(x_1, 1, i_{start})$ \\
 $T_x$.insert($(x_1, 1,i_{start})$, $(x_2, i_{start},1)$)\\
 $T_{\alpha}$.insert$(\alpha_1, \alpha_{i_{start}})$ \tcp{initialize $T_{\alpha}$}
\tcp{PART 2: Process further incoming edges}
 \For{$i=i_{start}..k$}{
    \tcp{Get closest $\alpha$ values}
     $\alpha_p, \alpha_q \longleftarrow$ \texttt{find-closest}($T_{\alpha}, \alpha_i$) \\
     $x_{p,q}\longleftarrow$ \texttt{compute-intersection}($f_p, f_q$) \\
    \tcp{Does $f_i$ outperform the previous functions at one or more intersections?}
    \uIf{$f_i(x_{p,q})<f_p(x_{p,q})$ or $f_i(x_{p,q})<f_q(x_{p,q})$}{
        $T_{\alpha}$.insert$(\alpha_i)$\\
        $T_x$.delete($(x_{p,q}, p,q)$) \\
        \tcp{Check counter-clockwise intersection}
        $(x_1, o,p)\longleftarrow (x_{p,q}, p,q)$.previous \\
        \While{$f_i(x_1) < f_p(x_1) $}{
            \tcp{delete and get next counter-clockwise}
            $T_x$.delete($(x_1, o,p)$)\\
            $T_{\alpha}$.delete($\alpha_p$) \\
            $(x_1, o,p) \longleftarrow (x_1, o,p).$previous\\
        }
        $p^{*} \longleftarrow p$ \\
        \tcp{Check clockwise intersection}
        $(x_2, q,r)\longleftarrow (x_{p,q}, p,q)$.next \\
        \While{$f_i(x_2)<f_q(x_2)$}{
            \tcp{delete and get next clockwise}
            $T_x$.delete($(x_2, q,r)$)\\
            $T_{\alpha}$.delete($\alpha_q$) \\
            $(x_2, q,r) \longleftarrow (x_2, q,r)$.next\\
        }
        $q^{*} \longleftarrow q$ \\
    }
    
    $x_1^{*} \longleftarrow $ \texttt{compute-intersection}($f_{p^*}, f_i$) \\
    $x_2^* \longleftarrow $ \texttt{compute-intersection}($f_i, f_{q^{*}}$) \\
    $T_x$.insert($(x_1^{*},p^{*},i)$,$(x_2^*, i,q^{*})$) \\
    $(x_1, o,p^{*})$.next = $(x_2^{*}, i,q^{*})$.previous = $(x_1^{*},p^{*}, i)$\\
    $(x_1^{*},p^{*},i)$.next = $(x_2,q^{*}, r)$.previous = $(x_2^{*}, i,q^{*})$\\
}
\tcp{PART 3: process outgoing edges}
 \For{$j=1..l$}{
     $(x_1, p,q),(x_2,q,r) \longleftarrow$ \texttt{find-closest}($T_x, \beta_j$)\\
    \tcp{update predecessor}
    $P[e^+_j] = e_q^-$ \\
 }
 \caption{Efficient angle cost update algorithm for convex angle cost functions}
 \label{alg:convex}
\end{algorithm}

}
\normalsize

In the following, we will describe several properties of the proposed method that will be necessary to proof the correctness and runtime efficiency of \autoref{alg:convex}. 

\begin{lemma}\label{firstloop}
Let $f_p$ and $f_q$ be functions for the incoming edges $e^-_p$ and $e^-_q$ as defined above, and let $D[e^-_p]\leq D[e^-_q]$ and $\alpha_p <180$ and $\alpha_p<\alpha_q<\alpha_p+180$, such that $f_q$ is equal to $f_p$ shifted in positive $x$-direction and in positive $y$-direction. Let $\hat{\alpha_p} =  \alpha_p+180$ be the point on the opposite side of $\alpha_p$ on the cycle.\\
If $\exists x:\ f_q(x) < f_p(x)$ then $\forall x^{'}\in [x, \hat{\alpha_p}]:\ \ f_q(x^{'}) < f_p(x^{'})$
\end{lemma}
\begin{proof}
We will use a property of convex monotonically increasing functions that can be seen as an increasing returns property: 
\begin{claim}
For a convex function $g$ and $a>b$, $a,b,c\geq 0$ it holds that 
\begin{align}\label{claimd}
    g(a+c)-g(a) \geq g(b+c)-g(b)
\end{align}
\end{claim}
\begin{proof}
The intuition is that if $g$ is differentiable its gradients are increasing and thus the difference in value becomes larger. To prove the claim, we use another well-known property of convex functions, namely if $f$ is convex and $f(0)\leq 0$ then 
\begin{align}\label{eq:1de}
    \forall\ d,c\geq 0:\ \ f(c)+f(d)\leq f(c+d)
\end{align}
We define $f(x) = g(x+b) - g(b)$. Since $f$ is only a shifted version of $g$, $f$ is convex as well and also note that $f(0) = g(b)-g(b) = 0$. We further set $d = a-b > 0$ and plug in $d$ in property~\ref{eq:1de} to yield 
\begin{gather*}
    f(d) + f(c) \leq f(d+c) \\
    \iff f(a-b) + f(c) \leq f(a-b+c) \\
    \iff g(a-b+b) - g(b) + g(c+b) - g(b) \leq g(a-b+c+b) -g(b) \\
    \iff g(a) - g(b) + g(b+c) \leq g(a+c)
\end{gather*}
which is equivalent to the claim.
\end{proof}

To continue the proof of lemma \ref{firstloop} we consider the point $x$ where $f_q(x) < f_p(x)$. Observe that due to $D[e^-_p]\leq D[e^-_q]$ and with $c_a$ monotonically increasing,
\begin{gather*}
    f_q(x) < f_p(x) \\
    \implies  D[e^-_q] + c_a(|x-\alpha_q|_{m}) < D[e^-_p] + c_a(|\alpha_p-x|_{m}) \\
    \implies |x-\alpha_q|_{m} < |\alpha_p-x|_{m}
\end{gather*}
Let $ x^{'}$ be any other point between $x$ and the opposite point of $\alpha_p$ as defined above ($x^{'}\in [x, \hat{\alpha_p}]$). Since $c_a$ is monotonically increasing, $f_p$ is maximal for $\hat{\alpha_p}$. We use the claim show that $f_q(x^{'})<f_p(x^{'})$ by setting $a = |x - \alpha_p|_{m}$ and $b = |x - \alpha_q|_{m}$ (leading to $a>b$), and further setting $c = |x - x^{'}|_{m}$:

\begin{gather*}
    f_q(x^{'}) \\
    = D[e^-_q] + c_a(|x^{'} - \alpha_q|_{m}) \\
    \underset{*}{\leq} D[e^-_q] + c_a(|x^{'} - x|_{m}+ |x - \alpha_q|_{m}) \\
    = D[e^-_q] + c_a(c+b) \\
    \underset{\ref{claimd}}{\leq} D[e^-_q] + c_a(c+a) - c_a(a) + c_a(b) \\
    = D[e^-_q] +  c_a\big(|x^{'} - x|_{m} + |x - \alpha_p|_{m}\big) - c_a\big(|x - \alpha_p|_{m}\big) + c_a\big(|x - \alpha_q|_{m}\big) \\
    \underset{**}{=} D[e^-_q] +  c_a\big(|x^{'} - \alpha_p|_{m}\big) - c_a\big(|x - \alpha_p|_{m}\big) + c_a\big(|x - \alpha_q|_{m}\big) \\
    = f_q(x) +  c_a\big(|x^{'} - \alpha_p|_{m}\big) - c_a\big(|x - \alpha_p|_{m}\big) \\
    \underset{f_q(x)<f_p(x)}{<} D[e^-_p] + c_a\big(|x - \alpha_p|_{m}\big) + c_a\big(|x^{'} - \alpha_p|_{m}\big) - c_a\big(|x - \alpha_p|_{m}\big) \\
    = f_p(x^{'})
\end{gather*}

(*) follows from the triangle inequality that also holds in the circular case. (**) is based on the condition that $x^{'}\in [x, \hat{\alpha_p}]$. \\In short, due to the property of increasing returns, if $f_q$ is lower than $f_p$ at any point it will remain lower until intersecting with $f_p$ from the other side and thus at least until $\hat{\alpha_p}$.
\end{proof}

In Lemma~\ref{firstloop} it was assumed that $\alpha_p<\alpha_q<\alpha_p+180$ which might seem very restrictive. The assumption was made to reduce complexity in the proof, but the statement also holds for the circular case. To see this, observe that given two functions $f_p$ and $f_q$, the $x$-axis can be rotated and flipped to fulfill the above condition. Since it is circular, this operation is always possible and does not change the shown property. The same holds for the following observation, that formally shows that the any new function must only be evaluated at the intersection point of the previous functions in order to determine whether it is superior anywhere.

\begin{lemma}\label{interbest}
Let $\alpha_i$ refer to an angle of an incoming edge between $\alpha_p$ and $\alpha_q$, such that in the simplified non-circular model it can be assumed that $\alpha_p < \alpha_i < \alpha_q$. Assume also that $D[e_i^-] \geq D[e_p^-]$ and $D[e_i^-] \geq D[e_q^-]$. Let $x_{p,q}$ denote the intersection of $f_p$ and $f_q$.\\
If $f_i(x_{p,q}) >f_p(x_{p,q}) = f_q(x_{p,q})$ then $\nexists x: f_i(x) < f_p(x) \wedge f_i(x) < f_q(x) $.
\end{lemma}
\begin{proof}
We again use an observation that was made in Lemma~\ref{firstloop}: If $D[e_i^-] \geq D[e_p^-]$, then $f_i(x) < f_p(x)$ implies $|x-\alpha_i|_{m} < |\alpha_p-x|_{m}$ and the same holds for $q$. Thus, if there exists an $x$ where $f_i(x) < f_p(x)$ and $f_i(x) < f_q(x)$ then it follows that $\alpha_p < x < \alpha_q$. Assume $\alpha_p < x \leq \alpha_i < \alpha_q$ (the other case is analogous). Due to Lemma~\ref{firstloop} if $f_i(x)< f_p(x)$ at $x$ then it remains lower until the opposite point $\hat{\alpha_p}$. Since the intersection $x_{p,q}$ is at $\hat{\alpha_p}$ at the latest (otherwise $f_q$ would not be better then $f_p$ anywhere), it can be concluded that $f_i$ outperforms $f_p$ also at the intersection.
\end{proof}

\begin{theorem}\label{theorem:correct}
\autoref{alg:convex} assigns the correct predecessor to each outgoing edge.
\end{theorem}
\begin{proof}
Let $e^-_{i}$ be the optimal predecessor for an outgoing edge $j^*$. In the first loop of \autoref{alg:convex} an incoming edge $e^-_{i}$ is only skipped if its value at $\hat{\alpha_1}$ is higher than for $e_1$. By Lemma~\ref{firstloop}, if the value is not better than $f_1$ at $\hat{\alpha_1}$, then it is larger everywhere. Consequently, any incoming edge that is skipped in the first loop of \autoref{alg:convex} is outperformed by $e^-_1$ and would thus not be the unique optimal predecessor of any outgoing edge.\\
Any $e^-_i$ that is not skipped in the PART 1 will be considered at some point in PART 2 in \autoref{alg:convex}. Its closest angles $\alpha_p$ and $\alpha_q$ are retrieved. Due to Lemma~\ref{interbest} it suffices to check whether $f_i$ is better at the intersection of $f_p$ and $f_q$, otherwise it can be discarded. Lemma~\ref{interbest} can be applied because if $\alpha_p$ and $\alpha_q$ were retrieved from $T_{\alpha}$, then they have already been processed earlier than $\alpha_i$ and consequently their distance $D$ must be lower.

Furthermore, if $f_i$ is indeed lower at the intersection, then there might be other intersections where $f_i$ outperforms previous incoming edges. Note that any of such points must lie between $\alpha_p$ and $\alpha_q$. \autoref{alg:convex} takes care of such special cases by checking the neighboring intersections in both directions. If at some point it is outperformed by another function, then it can not become better at any further point due to convexity. This procedure thus defines the current area of responsibility for edge $e^-_i$.

With further iterations, this area can be reduced in size or removed. Due to the conditions in \autoref{alg:convex}, it is only removed if a new incoming edge indexed with $\hat{i}$ obtains lower values at both intersections of $f_i$ with its neighboring functions. 
Then \autoref{alg:convex} would correctly compute the new intersections of $f_{\hat{i}}$ with $f_p$ and $f_q$ which replace the intersections of $f_i$. It would thus not be possible that $f_i$ is the optimal predecessor of any outgoing edge, because it was previously outperformed behind the borders of its area of responsibility by the respective intersecting functions, and now it is also outperformed in the area inbetween.

In the last for-loop (PART 3) in \autoref{alg:convex} concerning the outgoing edges, the closest intersections $(x_1, u,v)$ and $(x_2, v,w)$ to $\beta_{j^*}$ are retrieved from $T_x$. Note that by the procedure of \autoref{alg:convex} it is ensured that \texttt{find-closest} always yields triples of this form where $v$ appears in both triples and $(x_1, u,v).\text{next} = (x_2, v, w)$ and $(x_2, v, w).\text{previous} = (x_1, u,v)$. 
To see this, observe that the while-loops in \autoref{alg:convex} stops at two triples  $(x_2, q^*, r)$ and $(x_1,o, p^*)$. Since all intersections inbetween were deleted, it is correct that $(x_2, q^*, r)$ is assigned the new intersection with of $f_i$ with $f_{q^*}$ as the previous triple and equivalently $(x_1, o, p^*)$ is assigned the new intersection of $f_i$ with $f_{p^*}$.
\end{proof}

\begin{lemma}\label{computeintersection}
The \texttt{compute-intersection} operation can be approximated efficiently and sufficiently for \autoref{alg:convex}.
\end{lemma}
\begin{proof}
While the efficiency to compute an intersection between a function and its shifted version depends on the specific function, it is actually not necessary here to know the exact intersection. Crucially, in the end we are only interested in the values $\beta_1, \dots, \beta_l$, so it suffices to evaluate the functions at these values. Given two functions $f_p, f_q$ and their theoretical intersection $x_{p,q}$, we just aim to find $j^* = \argmin_j |x_{p,q} -\beta_j| $. For this purpose a similar procedure as the \texttt{find-closest} function (\autoref{alg:findclosest}) is utilized. First, all $\beta_j\ (j\in[1..l])$ are inserted in an AVL tree T ($\mathcal{O}\big(l \log l\big)$). For each call of \texttt{compute-intersection}, the tree is traversed in the following manner: Starting at the root node, $f_p$ and $f_q$ are evaluated at the $\beta$ value stored in the node. If $f_p(\beta)>f_q(\beta)$, the left child node is selected, otherwise the right child node. The intuition is that if $f_p(\beta)>f_q(\beta)$ then the intersection must be at a smaller $\beta$ value where $f_p$ takes lower values. During traversal of the tree, the optimal node with respect to the distance $|f_p(\beta)-f_q(\beta)|$ and its distance itself are memorized such that the closest $\beta$ is determined once a leaf node is reached.

However, one might object that the border case of the circularity is problematic in this case. Indeed, it might happen that the rightmost leaf in the tree is reached but the optimal $\beta$ value (the one closest to the intersection) is not the leftmost one. To prevent this case, we propose to build two trees with all $\beta_j$ inserted: One sorted and balanced tree with all $\beta_j$, and one with all $\hat{\beta_j}$ where $\hat{\beta_j} = (\beta_j+180)\text{ mod } 360$. After both trees have been traversed, the value that is closer to the intersection, i.e. where $|f_p(\beta)-f_q(\beta)|$ is lower, is selected. This ensures that the closest $\beta$ value is found, as it is not a border case in at least one of the trees.

We claim further that the discrete evaluations do not impact the proof of \autoref{theorem:correct}. 
Assume that $\alpha_p\leq \beta_{p,q} \leq x_{p,q}$ (the other side is again analogous). Note that there can not exist any $\hat{\beta}$ that satisfies $\alpha_p\leq \beta_{p,q} \leq \hat{\beta} \leq x_{p,q}$ because then 
$\hat{\beta}$ would have been selected as the intersection between $f_p$ and $f_q$. \\On the other hand, there can be $\beta^{'}$ with  $\alpha_p\leq \beta_{p,q} \leq x_{p,q} \leq \beta^{'} \leq \alpha_q$. Crucially, it is possible that $f_i(\beta^{'})< f_p(\beta^{'})$ and even $f_i(x_{p,q})< f_q(x_{p,q})$, but $f_i(\beta_{p,q})>f_q(\beta_{p,q})$ such that the new function $f_i$ is not better at the approximated intersection but still lower at the actual intersection. However, it has been shown in Lemma~\ref{firstloop} that in this case $f_i(x)<f_q(x)\ \forall x<\beta^{'}$. This is why in \autoref{alg:convex} it is checked whether \textit{both} $f_p$ and $f_q$ outperform $f_i$ at the approximate intersection (which would be redundant if the exact intersection was available because the values for $f_p$ and $f_q$ would be equal). Thus, if $\beta^{'}$ with the above properties exists, then $f_i(\beta_{p,q})<f_q(\beta_{p,q})$ and therefore the algorithm would still compute the intersections of $f_i$ with the previous functions and $\beta^{'}$ would be assigned to the correct range.
\end{proof}

\begin{theorem}
The runtime of \autoref{alg:convex} is $\mathcal{O}\big((k+l) \log kl \big)$.
\end{theorem}
\begin{proof}
First, the incoming edges are sorted by their distance to the source vertex, taking $\mathcal{O}\big(k \log k \big)$ time. 
As argued in lemma \ref{computeintersection}, the \texttt{compute-intersection} operation can be approximated here by a discrete evaluation on the $\beta$-values, leading to a runtime of $\mathcal{O}\big(\log l \big)$ for each call to \texttt{compute-intersection} once the AVL-tree of $\beta$-values is build ($\mathcal{O}\big(l \log l \big)$). In addition, each insert and delete operation to $T_x$ takes $\mathcal{O}\big(\log k \big)$ since at most $k$ elements are in the tree (in each iteration, one intersection is deleted and two intersections are added). Similarly, each insert and delete operation, as well as each call to \texttt{find-closest} on $T_{\alpha}$ requires $\mathcal{O}\big(\log k \big)$ time. Last, the \texttt{find-closest} operation again only requires a traversal of $T_x$ in $\mathcal{O}\big(\log k \big)$ steps.\\
Now it is important to observe that in the while-loop of PART 2 of \autoref{alg:convex}, elements are only deleted from $T_x$ and $T_a$. The number of operations in the while loop is therefore limited to the maximum number of elements of the tree. The tree contains at most $\mathcal{O}(k)$ values at any point, because in each iteration (for each incoming edge) at most two one intersections are added, and one is deleted. Thus, PART 2 of \autoref{alg:convex} requires $\mathcal{O}\big(k (\log l + \log k) \big) = \mathcal{O}\big(k \log kl \big)$ operations.

Finally, in PART 3 of \autoref{alg:convex} the closest element in $T_x$ is found for each outgoing edge, causing an additional runtime of $\mathcal{O}\big(l \log k \big)$. Given the optimal predecessor edges, the distances of $E^+$ can be updated in $\mathcal{O}\big(l \big)$. Altogether, the runtime is $\mathcal{O}\big(k \log k + l \log l + k \log kl + l \log k \big) \in \mathcal{O}\big((k+l) \log kl\big)$. 
\end{proof}